\documentclass[11pt]{amsart}   	
\usepackage{geometry}      
\usepackage[dvipsnames]{xcolor}   		             		
\usepackage{graphicx}				
\usepackage{amssymb,mathrsfs}
\usepackage{amsthm,stmaryrd}
\usepackage[leqno]{amsmath}
\usepackage{tikz}
\usepackage{tikz-cd}
\usepackage{accents,upgreek,enumerate}
\usepackage[headings]{fullpage}
\usepackage{bm}
\usepackage{mathtools}
\usepackage[all]{xy}
\usepackage{caption}
\usepackage{comment}
\usepackage{tabu}

\tikzset{
  commutative diagrams/.cd, 
  arrow style=tikz, 
  diagrams={>=stealth}
}

  \makeatletter
\def\@tocline#1#2#3#4#5#6#7{\relax
  \ifnum #1>\c@tocdepth 
  \else
    \par \addpenalty\@secpenalty\addvspace{#2}%
    \begingroup \hyphenpenalty\@M
    \@ifempty{#4}{%
      \@tempdima\csname r@tocindent\number#1\endcsname\relax
    }{%
      \@tempdima#4\relax
    }%
    \parindent\z@ \leftskip#3\relax \advance\leftskip\@tempdima\relax
    \rightskip\@pnumwidth plus4em \parfillskip-\@pnumwidth
    #5\leavevmode\hskip-\@tempdima
      \ifcase #1
       \or\or \hskip 1em \or \hskip 2em \else \hskip 3em \fi%
      #6\nobreak\relax
    \dotfill\hbox to\@pnumwidth{\@tocpagenum{#7}}\par
    \nobreak
    \endgroup
  \fi}
\makeatother

\usetikzlibrary{calc}
\usetikzlibrary{fadings}
\usetikzlibrary{decorations.pathmorphing}
\usetikzlibrary{decorations.pathreplacing}

\newcounter{marginnote}
\DeclareMathAlphabet{\mathpzc}{OT1}{pzc}{m}{it}

\usepackage[backref=page]{hyperref}
\hypersetup{
  colorlinks   = true,          
  urlcolor     = violet,          
  linkcolor    = blue,          
  citecolor   = violet             
}

\usepackage{mathrsfs}
    
\usepackage{appendix}
\usepackage{amsmath}
\usepackage{amsfonts}
\usepackage{amssymb}
\usepackage{hyperref}
\usepackage{amsthm}
\usepackage{marginnote}
\usepackage{stmaryrd}
\usepackage{enumitem}
\usepackage[english]{babel}
\usepackage{yfonts}
\usepackage[utf8x]{inputenc}
\usepackage{verbatim}
\usepackage{graphicx}
\usepackage{faktor}
\usepackage{xcolor}
\usepackage{xfrac}
\usepackage{tikz,tikz-cd}
\usetikzlibrary{decorations.pathmorphing,decorations.pathreplacing,patterns}
\usepackage[all]{xy}
\usepackage{bbm}
\usepackage{tabularx}
\usepackage{longtable}
\usepackage{tabu}
\usepackage{booktabs}
\usepackage{mathtools}

\usepackage[]{textcomp}
\usepackage[varqu,varl]{inconsolata}
\usepackage[cal=cm]{mathalfa}

\newcommand{\TT}{\operatorname{T}}
\newcommand{\M}[4]{\overline{\mathcal{M}}_{#1,#2}(#3,#4)}
\newcommand{\Q}[4]{\mathcal{Q}_{#1,#2}(#3,#4)}
\newcommand{\Qe}[4]{\mathcal{Q}^{\epsilon}_{#1,#2}(#3,#4)}

\newcommand{\QG}[4]{\mathcal{Q}G_{#1,#2}(#3,#4)}
\newcommand{\QGe}[4]{\mathcal{Q}G^{\epsilon}_{#1,#2}(#3,#4)}

\newcommand{\PP}{\mathbb P}

\newcommand{\Z}{\mathbb{Z}}
\newcommand{\N}{\mathbb{N}}
\newcommand{\OO}{\mathcal{O}}
\renewcommand{\to}{\rightarrow}

\newcommand{\EE}{\mathbf{E}}

\newcommand{\LL}{\mathbf{L}}
\newcommand{\MM}{\mathfrak M}
\newcommand{\Aaff}{\mathbb{A}}

\newcommand{\comp}{\chi}

\newcommand{\Pic}{\operatorname{Pic}}

\newcommand{\Hom}{\operatorname{Hom}}

\newcommand{\Gm}{\mathbb{G}_{\text{m}}}
\newcommand{\virt}[1]{[#1]^{\operatorname{virt}}}

\newcommand{\Id}{\operatorname{Id}}
\newcommand{\CC}{\mathbb{C}}
\newcommand{\QQ}{\mathbb{Q}}
\newcommand{\HH}{\mathrm{H}}
\newcommand{\Achow}{\operatorname{A}}
\newcommand{\pt}{\operatorname{pt}}
\newcommand{\bq}{\begin{equation}}
\newcommand{\eq}{\end{equation}}
\newcommand{\ba}{\begin{aligned}}
\newcommand{\ea}{\end{aligned}}
\newcommand{\be}{\begin{enumerate}}
\newcommand{\ee}{\end{enumerate}}
\newcommand{\bsm}{\left(\begin{smallmatrix}}
\newcommand{\esm}{\end{smallmatrix}\right)}                   
\newcommand{\bpm}{\begin{pmatrix}}
\newcommand{\epm}{\end{pmatrix}}
\newcommand{\barr}{\begin{displaymath}\begin{array}{cccc}}
\newcommand{\earr}{\end{array}\end{displaymath}}
\newcommand{\barrl}{\begin{displaymath}\begin{array}{lcl}}
\newcommand{\earrl}{\end{array}\end{displaymath}}
\newcommand{\barl}{\begin{displaymath}\begin{array}{l}}
\newcommand{\earl}{\end{array}\end{displaymath}}
\newcommand{\bxym}{ \begin{displaymath}\xymatrix }
\newcommand{\exym}{\end{displaymath}}
\newcommand{\bcd}{\begin{center}\begin{tikzcd}}
\newcommand{\ecd}{\end{tikzcd}\end{center}}
\newcommand{\R}{\operatorname{R}^{\bullet}}

\newcommand{\pr}{\operatorname{pr}}
\newcommand{\ev}{\operatorname{ev}}
\newcommand{\codim}{\operatorname{codim}}
\newcommand{\vdim}{\operatorname{vdim}}

\newcommand{\om}[1]{\mathcal{#1}}

\newcommand{\NN}{\operatorname{N}}

\theoremstyle{definition}
\newtheorem{thm}{Theorem}[subsection]
\newtheorem{lem}[thm]{Lemma}
\newtheorem{prop}[thm]{Proposition}
\newtheorem{cor}[thm]{Corollary}
\newtheorem*{teo*}{Theorem}

\newtheorem{innercustomthm}{Theorem}
\newenvironment{customthm}[1]
  {\renewcommand\theinnercustomthm{#1}\innercustomthm}
  {\endinnercustomthm}

\theoremstyle{definition}

\newtheorem{dfn}[thm]{Definition}
\newtheorem{aside}[thm]{Aside}
\newtheorem{remark}[thm]{Remark}

\newtheorem*{aside*}{Aside}
\newtheorem*{acknowledgements}{Acknowledgements}

\newcommand{\ilemph}[1]{\emph{#1}}

\begin{document}

\title{Relative quasimaps and mirror formulae}
\author{Luca Battistella and Navid Nabijou}

\begin{abstract}We construct and study the theory of relative quasimaps in genus zero, in the spirit of Gathmann. When $X$ is a smooth toric variety and $Y$ is a smooth very ample hypersurface in $X$, we produce a virtual class on the moduli space of relative quasimaps to $(X,Y)$, which we use to define relative quasimap invariants. We obtain a recursion formula which expresses each relative invariant in terms of invariants of lower tangency, and apply this formula to derive a quantum Lefschetz theorem for quasimaps, expressing the restricted quasimap invariants of $Y$ in terms of those of $X$. Finally, we show that the relative $I$-function of Fan--Tseng--You coincides with a natural generating function for  relative quasimap invariants, providing mirror-symmetric motivation for the theory.\end{abstract}

\maketitle

\section{Introduction}
\subsection{The aim of relative quasimap theory} Relative Gromov--Witten theory for smooth pairs $(X,Y)$ occupies a central place in modern enumerative geometry, owing both to its intrinsic interest and to the role it plays in the degeneration formula \cite{Ga,Vre,Li1,Li2,MaulikPandharipande}. However, while the structures underlying the absolute Gromov--Witten invariants --- generating functions, Frobenius manifolds, Lagrangian cones, etc. --- are well-understood, the nature of the corresponding structures in the relative setting remains mysterious.

A promising avenue for addressing these questions is to adapt ideas coming from mirror symmetry to the relative context. There are several possible approaches here, depending on one's viewpoint on mirror symmetry, but the end goal of each should be to obtain pleasant closed formulae for generating functions of relative Gromov--Witten invariants.

Recent work of Fan--Tseng--You \cite{FanTsengYou} uses the correspondence between relative invariants and the Gromov--Witten invariants of orbifolds \cite{AbramovichCadmanWise} in order to derive a mirror theorem for certain restricted generating functions of relative invariants.

We propose a more direct approach to the problem. Our motivation comes from the theory of stable quasimaps: this theory dates back to Givental's earliest work on mirror symmetry \cite{Givental-mirror}, and since then has been systematised and extended in order to prove a large class of mirror theorems in the absolute setting \cite{CF-K,CFKM, CF-K-wallcrossing,CF-K-MirrorSymmetry}. The belief is that a fully-fledged theory of \emph{relative} quasimaps should lead to an equally powerful collection of mirror theorems in the relative setting.

\subsection{Results}
In this paper, we realise this proposal in the context of genus-zero quasimaps relative to a hyperplane section.

\subsubsection{Construction and recursion} We begin by constructing moduli spaces of relative stable quasimaps in the spirit of Gathmann, that is, as substacks of moduli spaces of (absolute) quasimaps:
\begin{equation*} \mathcal{Q}_{0,\alpha}(X|Y,\beta) \hookrightarrow \mathcal{Q}_{0,n}(X,\beta).\end{equation*}
These spaces are equipped with natural virtual fundamental classes, and hence can be used to define \emph{relative quasimap invariants}. We prove a recursion formula for such invariants, expressing each relative invariant in terms of invariants with smaller numerical data (see Theorem \ref{thm recursion intro} below).

\subsubsection{Application 1: wall-crossing} We then apply the recursion formula to relate our invariants to the relative $I$-function considered in \cite{FanTsengYou}. This demonstrates that the theory of relative quasimaps provides a natural framework for studying relative mirror symmetry.
\begin{customthm}{A}[\textbf{Theorem \ref{wallcrossing thm}}]\label{thmA} The relative $I$-function coincides with the following natural generating function for relative quasimap invariants
\begin{equation*}I^{X|Y}(q,z,0) = \sum_\beta q^\beta (\ev_1)_*\left( \frac{1}{z-\psi_1}\virt{\Q{0}{(Y\cdot\beta,0)}{X|Y}{\beta}}\right).\end{equation*}
\end{customthm}
\noindent This result provides a geometric interpretation of the combinatorial factors appearing in \cite[Theorem 4.3]{FanTsengYou}: they are the first Chern classes of the tangency line bundles appearing in the recursion formula for relative quasimaps. The fact that no other terms enter into the recursion is due to the stronger stability condition enjoyed by the quasimap spaces.

\subsubsection{Application 2: quasimap quantum Lefschetz} Besides the connections to relative Gromov--Witten theory, we may also apply the recursion formula to study absolute quasimap theory. By repeatedly decreasing the tangency orders, we obtain a quantum Lefschetz theorem for quasimap invariants, expressing the invariants of a hyperplane section $Y$ in terms of those of $X$. This takes two forms; first, we have a general result which holds without any special restrictions on the target geometry:
\begin{customthm}{B}[\textbf{Theorem \ref{Theorem full quasimap Lefschetz}}] \label{thmB} Let $X$ be a smooth projective toric variety and $Y \subseteq X$ a smooth very ample hypersurface. Then there is an explicit algorithm to recover the (restricted) quasimap invariants of $Y$, as well as the relative invariants of $(X,Y)$, from the quasimap invariants of $X$.\end{customthm}

Even in the situation where the quasimap theories of $X$ and $Y$ coincide with the (respective) Gromov--Witten theories, the quasimap algorithm is much more efficient than the Gromov--Witten algorithm, due to the absence of rational tails. It is clear that this phenomenon generalises to the setting of the quasimap degeneration formula (once such a result has been established); we plan to revisit this in future work.

The second form of the quantum Lefschetz theorem is an explicit relation between generating functions, obtained by applying Theorem \ref{thmB} in the semipositive context.
\begin{customthm}{C}[\textbf{Theorem \ref{Theorem Quantum Lefschetz}}]\label{thmC}
Let $X$ be a smooth projective toric Fano variety and let $i\colon Y \hookrightarrow X$ be a very ample hypersurface. Assume that $-K_Y$ is nef and that $Y$ contains all curve classes (see \S \ref{Subsection setup}). Then
\begin{equation*}
\tilde{S}_0^Y(z,q) = \dfrac{\sum_{\beta\geq 0} q^\beta\left(\prod_{j=0}^{Y\cdot\beta}(Y+jz)\right)\cdot \mathscr{S}^X_\beta(z)}{P_0^X(q)}
\end{equation*}
where $\tilde{S}_0^Y(z,q)$ and $\mathscr{S}^X_\beta(z)$ are the following generating functions for $2$-pointed quasimap invariants
\begin{align*}
\tilde{S}_0^Y(z,q)&=i_* \sum_{\beta \geq 0} q^\beta (\ev_1)_*\left(\frac{1}{z-\psi_1} \virt{\Q{0}{2}{Y}{\beta}}\right)\\
\mathscr{S}^X_\beta(z)&=(\ev_1)_*\left(\frac{1}{z-\psi_1} \virt{\Q{0}{2}{X}{\beta}}\right)
\end{align*}
and $P_0^X(q)$ is given by:
$$ P_0^X(q) = 1 + \sum_{\substack{\beta>0 \\ K_Y\cdot\beta=0}} q^\beta(Y\cdot\beta)!\langle [\pt_X] \psi_1^{Y\cdot\beta-1} ,\mathbbm 1_{X}\rangle_{0,2,\beta}^X.$$\end{customthm}
\noindent This is similar in spirit to the results of \cite{Ga-MF}; however, the stronger stability condition considerably simplifies both the proof and the final formula. This result can also be obtained as a consequence of \cite[Corollary~5.5.1]{CF-K-wallcrossing}; see \S \ref{Subsection CFK comparison}.

\subsection{Plan of the paper}
The plan of the paper is as follows. In \S \ref{Section relative stable quasimaps} we define the space of relative quasimaps, as a substack of the moduli space of (absolute) quasimaps:
\begin{equation*} \Q{0}{\alpha}{X|Y}{\beta} \hookrightarrow \Q{0}{n}{X}{\beta}. \end{equation*}
Here $X$ is a smooth toric variety, $Y$ is a smooth very ample hypersurface and $\alpha = (\alpha_1, \ldots, \alpha_n)$ encodes the orders of tangency of the marked points to $Y$. Note that we \emph{do not} require $Y$ to be toric.

The study of the relative geometry $(\PP^N,H)$, for $H \subseteq \PP^N$ a hyperplane, plays a fundamental role: in this case, the relative space is irreducible of the expected codimension $\Sigma_{i=1}^n \alpha_i$ (in fact, it is the closure of the so-called ``nice locus'' consisting of maps from a $\PP^1$ whose image is not contained inside $H$, and which satisfies the tangency conditions at the marked points). Thus it has an actual fundamental class, which we use to define relative quasimap invariants of the pair $(\PP^N,H)$. This fundamental class can also be pulled back to endow $\Q{0}{\alpha}{X|Y}{\beta}$ with a virtual  class, and hence we can define relative quasimap invariants of a general pair $(X,Y)$ as above.

In \S \ref{Section recursion for PN} we prove a recursion formula relating $[\Q{0}{\alpha}{\PP^N|H}{d}]$ with classes determined by lower numerical invariants. We make use of the \emph{comparison morphism}:
\begin{equation*} \comp : \M{0}{n}{\PP^N}{d} \to \Q{0}{n}{\PP^N}{d} \end{equation*} 
This restricts to a birational morphism between the relative spaces, which we use to push down Gathmann's formula to obtain a recursion formula for relative stable quasimaps.

In \S \ref{Section recursion formula in general case} we turn to the case of an arbitrary pair $(X,Y)$ with $Y$ very ample. We use the embedding $X \hookrightarrow \PP^N$ defined by $\OO_X(Y)$ to construct a virtual class $\virt{\Q{0}{\alpha}{X|Y}{\beta}}$.
We then prove the recursion formula for $(X,Y)$ by pulling back the formula for $(\PP^N,H)$. This requires several comparison theorems for virtual classes, extending results in Gromov--Witten theory to the setting of quasimaps. The full statement of the recursion formula is:

\begin{customthm}{D}[\textbf{Theorem \ref{Theorem general recursion}}] \label{thm recursion intro} Let $X$ be a smooth projective toric variety and let $Y \subseteq X$ be a very ample hypersurface (not necessarily toric). Then
\begin{equation*} (\alpha_k \psi_k + \ev_k^* [Y]) \cap \virt{\Q{0}{\alpha}{X|Y}{\beta}} = \virt{\Q{0}{\alpha+e_k}{X|Y}{\beta}} + \virt{\mathcal D^\mathcal{Q}_{\alpha,k}(X|Y,\beta)} \end{equation*}
in the Chow group of $\Q{0}{\alpha}{X|Y}{\beta}$. 
\end{customthm}
\noindent Here $\mathcal D^\mathcal{Q}_{\alpha,k}(X|Y,\beta)$ is a certain \emph{quasimap comb locus} sitting inside the boundary of the relative space; its virtual class should be thought of as a correction term. Such terms also appear in Gathmann's stable map recursion formula; however, in our setting the stronger stability condition for quasimaps considerably reduces the number of such contributions.

In \S \ref{Section quasimap mirror theorem} we apply Theorem \ref{Theorem general recursion} to prove the quasimap quantum Lefschetz theorems \ref{thmB} and \ref{thmC} discussed above. Finally in \S \ref{section wallcrossing} we apply the same result to prove Theorem \ref{thmA}, equating the relative $I$-function with a generating function for our relative quasimap invariants.

In Appendix~\ref{appendix:intersection} we define the \emph{diagonal pull-back} along a morphism whose target is smooth, and verify that it agrees with the more modern concept of virtual pull-back \cite{Manolache-Pull} when both are defined. The diagonal pull-back was employed implicitly in \cite{Ga}, but we find it useful here to give a more explicit treatment.

\subsection{Future directions}
We expect Theorem \ref{thmA} to extend to general simple normal crossings divisors and arbitrary genus. To this end, we are currently in the process of developing and applying a fully-fledged theory of \emph{logarithmic quasimaps}, taking inspiration from the theory of logarithmic stable maps \cite{GrossSiebertLog,ChenLog,AbramovichChenLog}. In this context, the ``relative mirror theorem'' will take the form of a wall-crossing formula between logarithmic quasimap and Gromov--Witten invariants, together with closed-form expressions for the quasimap generating functions.

\begin{acknowledgements}
We thank Tom Coates, Cristina Manolache and Dhruv Ranganathan for helpful advice concerning both exposition and subject matter. We are especially grateful to Ionu\c{t}~Ciocan-Fontanine for thoroughly examining a preliminary version of this manuscript and providing several clarifications and corrections. We would also like to thank Fabio Bernasconi, Andrea Petracci and Richard Thomas for useful conversations. We thank the referees for a number of valuable suggestions.

This work was supported by EPSRC, the Royal Society and the London School of Geometry and Number Theory.
\end{acknowledgements}

\subsection{Table of notation} We will use the following notation, most of which is introduced in the main body of the paper.\medskip

\tabulinesep=^0.4mm_0.4mm
\begin{longtabu}{r c c p{0.8\linewidth}}
$X$ & & & a smooth projective toric variety \\
$Y$ & & & a smooth very ample hypersurface in $X$ \\
$\Sigma$, $\Sigma(1)$ & & & the fan of $X$, and the set of $1$-dimensional cones of $\Sigma$ \\
$\rho$, $D_\rho$ & & & an element of $\Sigma(1)$, and the toric divisor in $X$ associated to it \\
$\M{g}{n}{X}{\beta}$ & & & the moduli space of stable maps to $X$ \\
$\M{0}{\alpha}{X|Y}{\beta}$ & & & the moduli space of relative stable maps to $(X,Y)$; see \S \ref{subsection relative maps} \\
$\Q{g}{n}{X}{\beta}$ & & & the moduli space of toric quasimaps to $X$; see \S \ref{subsection quasimaps} \\
$\mathcal{Q}^{\circ}_{0,\alpha}(X|Y,\beta)$ & & & the nice locus of relative quasimaps to $(X,Y)$; see \S \ref{Subsection basic properties of the moduli space} \\
$\Q{0}{\alpha}{X|Y}{\beta}$ & & & the moduli space of relative quasimaps to $(X,Y)$; see \S \ref{Subsection relative stable quasimaps} \\
$\mathcal{D}^{\mathcal{Q}}_{\alpha,k}(X|Y,\beta)$ & & & the quasimap comb locus; see \S \ref{Subsection recursion formula for PN} \\
$\mathcal{D}^{\mathcal{Q}}(X|Y,A,B,M)$ & & & (a component of) the comb locus; see \S \ref{Subsection recursion formula for PN} \\
$\mathcal{E}^{\mathcal{Q}}(X|Y,A,B,M)$ & & & the total product for the comb locus; see \S \ref{Subsection recursion formula for PN} \\
$\mathcal{D}^{\mathcal{Q}}(X,A,B)$ & & & the quasimap centipede locus; see \S \ref{Subsection recursion formula for PN} \\
$\mathcal{E}^{\mathcal{Q}}(X,A,B)$ & & & the total product for the centipede locus; see \S \ref{Subsection recursion formula for PN} \\
$\MM^{\operatorname{wt}}_{g,n}$ & & & the moduli stack of weighted prestable curves; see \S \ref{Subsection recursion formula for PN} \\
$\mathfrak{Bun}^{G}_{g,n}$ & & & the moduli stack of principal $G$-bundles on the universal curve over $\MM_{g,n}$; see Remark \ref{Section comparison with GIT construction} \\
$\om{Q}(f)$ & & & the push-forward morphism between quasimap spaces; see \S \ref{Subsection relative stable quasimaps} \\
$\chi$ & & & the comparison morphism from stable maps to quasimaps; see \S \ref{Subsection basic properties of the moduli space} \\
$f^!_{\text{v}}$ & & & virtual pull-back for $f$ virtually smooth; see Appendix \ref{appendix:intersection} \\
$f^!_{\Delta}$ & & & diagonal pull-back; see Appendix \ref{appendix:intersection}
\end{longtabu}
\section{Relative stable quasimaps} \label{Section relative stable quasimaps}
We begin with a brief recollection of the theories of stable quasimaps and relative stable maps, thus putting our work in its proper context.
\subsection{Stable quasimaps}\label{subsection quasimaps}
The moduli space of \ilemph{stable toric quasimaps} $\Q{g}{n}{X}{\beta}$ was constructed by I. Ciocan-Fontanine and B. Kim \cite{CF-K} as a compactification of the moduli space of smooth curves in a smooth and complete toric variety $X$. Roughly speaking, the objects are rational maps $C \dashrightarrow X$ where $C$ is a nodal curve, subject to a stability condition. The precise definition depends on the description of $X$ as a GIT quotient; see \cite[Definition~3.1.1]{CF-K} and \cite[Definition~3.1.1]{CFKM}.  The space $\Q{g}{n}{X}{\beta}$ is a proper Deligne--Mumford stack of finite type.  It admits a virtual fundamental class, which is used to define curve-counting invariants for $X$ called \ilemph{quasimap invariants}.

This theory agrees with that of stable quotients \cite{MOP} when both are defined, namely when $X$ is a projective space.  There is a common generalisation given by the theory of stable quasimaps to GIT quotients \cite{CFKM}. For simplicity, we will work mostly in the toric setting; however, this restriction is not essential for our arguments. Thus in this paper when we say ``quasimaps'' we are implicitly talking about toric quasimaps.
Quasimap invariants provide an alternative system of curve counts to the more well-known Gromov--Witten invariants. These latter invariants are defined via moduli spaces of stable maps, and as such we will often refer to them as \emph{stable map invariants}.

For $X$ sufficiently positive, the quasimap invariants coincide with the Gromov--Witten invariants, in all genera. This has been proven in the following cases:
\begin{itemize}[leftmargin=0.7cm]
\item $X$ a projective space or a Grassmannian: see \cite[Theorems~ 3~and~4]{MOP}, and \cite{ManolacheStable} for an alternative proof.
\item $X$ a projective complete intersection of Fano index at least $2$: see \cite[Corollary 1.7]{CF-K-MirrorSymmetry}, and \cite{CZ-mirror} for an earlier approach.
\item $X$ a projective toric Fano variety: see \cite[Corollary 1.3]{CF-K-higher-genus}.
\end{itemize}
In general, however, the invariants differ, the difference being encoded by certain wall-crossing formulae, which can be interpreted in the context of toric mirror symmetry \cite{CF-K-wallcrossing}.

\subsection{Relative stable maps}\label{subsection relative maps}
Let $Y$ be a smooth very ample hypersurface in a smooth projective variety $X$. In \cite{Ga} A. Gathmann constructs a space of relative stable maps to the pair $(X,Y)$ as a closed substack of the moduli space of (absolute) stable maps to $X$:
The relative space parametrises stable maps with prescribed tangencies to $Y$ at the marked points.  Unfortunately this space does not admit a natural perfect obstruction theory. Nevertheless, because $Y$ is very ample it is still possible to construct a virtual fundamental class by intersection-theoretic methods, and hence one can define relative stable map invariants.
Gathmann establishes a recursion formula for these virtual classes which allows one to express any relative invariant of $(X,Y)$ in terms of absolute invariants of $Y$ and relative invariants with lower contact multiplicities. By successively increasing the contact multiplicities from zero to the maximum possible value, this gives an algorithm expressing the (restricted) invariants of $Y$ in terms of those of $X$: see \cite[Corollary 5.7]{Ga}. In \cite{Ga-MF} this result is applied to give an alternative proof of the mirror theorem for projective hypersurfaces \cite{Givental-equivariantGW} \cite{LLY1}.

\subsection{Definition of relative stable quasimaps} \label{Subsection relative stable quasimaps}

For the rest of the paper, $X$ will denote a smooth projective toric variety and $Y \subseteq X$ a smooth very ample hypersurface. We \emph{do not} require that $Y$ is toric.
Consider the line bundle $\OO_X(Y)$ and the section $s_Y$ cutting out $Y$. By \cite{CoxRing} we have a natural isomorphism of $\CC$-vector spaces
\begin{equation*} \HH^0(X,\OO_X(Y)) = \left\langle \prod_{\rho} z_\rho^{a_\rho} : \Sigma_\rho a_\rho [D_\rho] = [Y] \right\rangle_{\CC} \end{equation*}
where the $z_\rho$ for $\rho \in \Sigma(1)$ are the generators of the Cox ring of $X$ and the $a_\rho$ are non-negative integers. We can therefore write $s_Y$ as
\begin{equation*} s_Y = \sum_{\underline{a}=(a_\rho)} \lambda_{\underline{a}} \prod_\rho z_\rho^{a_\rho} \end{equation*}
where the $\underline{a} = (a_\rho) \in \N^{\Sigma(1)}$ are exponents and the $\lambda_{\underline{a}}$ are scalars. The idea is that a quasimap
\begin{equation*} \big((C,x_1,\ldots,x_n), (L_\rho,u_\rho)_{\rho \in \Sigma(1)}, (\varphi_m)_{m \in M}\big) \end{equation*}
should ``map'' a point $x \in C$ into $Y$ if and only if the section
\begin{equation} \label{uY expression} u_Y := \sum_{\underline{a}} \lambda_{\underline{a}} \prod_\rho u_\rho^{a_\rho} \end{equation}
vanishes at $x$. We now explain how to make sense of expression \eqref{uY expression}. For each exponent $\underline{a}$ appearing in $s_Y$ we have a well-defined section:
\begin{equation*} u_{\underline{a}} := \lambda_{\underline{a}} \prod_\rho u_\rho^{a_\rho} \in \HH^0(C,\otimes_\rho L_\rho^{\otimes a_\rho}) \end{equation*}
Furthermore, given two such $\underline{a}$ and $\underline{b}$, since $\sum_\rho a_\rho [D_\rho] = [Y] = \sum_\rho b_\rho [D_\rho]$ in $\Pic{X}$ it follows from the exact sequence
\begin{equation*} 0 \to M \to \Z^{\Sigma(1)} \to \Pic{X} \to 0 \end{equation*}
that $\underline{a}$ and $\underline{b}$ differ by an element $m$ of $M$. Thus the isomorphism $\varphi_m$ allows us to view the sections $u_{\underline{a}}$ and $u_{\underline{b}}$ as sections of the same bundle, which we denote by $L_Y$ (there is a choice for $L_Y$ here, but up to isomorphism it does not matter). We can thus sum the $u_{\underline{a}}$ together to obtain $u_Y$.

The upshot is that we obtain a line bundle $L_Y$ on $C$, which plays the role of the ``pull-back'' of $\OO_X(Y)$ along the ``map'' $C \to X$, and a global section
\begin{equation*} u_Y \in \HH^0(C,L_Y) \end{equation*}
which plays the role of the ``pull-back'' of $s_Y$. With this at hand, we are ready to give the main definition of this paper. We begin with the case $X=\PP^N$ and $Y=H=\{ z_0 = 0\}$ a co-ordinate hyperplane. In this situation, the discussion above simplifies significantly.
\begin{dfn} Fix a number $n\geq 2$ of marked points, a degree $d\geq 0$ and a vector $\alpha=(\alpha_1, \ldots, \alpha_n)\in\mathbb N^n$ of tangency orders such that $\Sigma_i \alpha_i \leq d$.
 The space of \emph{relative stable quasimaps} $\Q{0}{\alpha}{\PP^N|H}{d}$ is the closure inside $\Q{0}{n}{\PP^N}{d}$ of the so-called \emph{nice locus}, consisting of quasimaps with smooth source curve
 \begin{equation*} \left( (\PP^1,x_1,\ldots,x_n),u_0,\ldots,u_N \right), \qquad u_i \in \HH^0(\PP^1,\OO_{\PP^1}(d))\end{equation*}
 such that:
 \begin{enumerate}
 	\item $u_0 \not\equiv 0$;
 	\item the relation $u_0^*(0)\geq\sum_i \alpha_ix_i$ holds in $\Achow_*(u_0^{-1}(0))$;
 	\item the sections $(u_0,\ldots,u_N)$ do not vanish simultaneously on $\PP^1$ (that is, there are no basepoints).
 \end{enumerate} \end{dfn}
\noindent In the above definition $u_0^*(0) := 0^!([\PP^1])\in\Achow_*(u_0^{-1}(0))$ is obtained via Fulton's refined Gysin map \cite[\S 2.6]{FUL} applied to the diagram:
\bcd
u_0^{-1}(0) \ar[r] \ar[d] \ar[rd,phantom,"\square"] & \PP^1 \ar[d,"u_0"] \\
\PP^1 \ar[r,"0"] & \OO_{\PP^1}(d)
\ecd

In the general case, the complete linear system $|\OO_X(Y)|$ defines an embedding $i \colon X \hookrightarrow \PP^N$ such that $i^{-1}(H) = Y$ for some hyperplane $H$. By the functoriality property of quasimap spaces (see \cite[\S 3.1]{CF-K-wallcrossing}) we have a map
\begin{equation*} k = \om{Q}(i) \colon \Q{0}{n}{X}{\beta} \to \Q{0}{n}{\PP^N}{d} \end{equation*}
where $d=i_*\beta$. 
\begin{dfn} As before, fix a number $n \geq 2$ of marked points, an effective curve class $\beta \in \HH_2^+(X)$ and a vector $\alpha=(\alpha_1,\ldots,\alpha_n)\in\mathbb N^n$ of tangency orders such that $\Sigma_i\alpha_i\leq Y\cdot\beta$. The \ilemph{space of relative stable quasimaps}
\begin{equation*} \Q{0}{\alpha}{X|Y}{\beta} \subseteq \Q{0}{n}{X}{\beta} \end{equation*}
is defined as the following fibre product:
\bcd
\Q{0}{\alpha}{X|Y}{\beta}\ar[d]\ar[r,hook]\ar[dr,phantom,"\Box"] & \Q{0}{n}{X}{\beta}\ar[d,"k"] \\
\Q{0}{\alpha}{\PP^N|H}{d}\ar[r,hook] & \Q{0}{n}{\PP^N}{d}
\ecd
\end{dfn}

\begin{remark}\label{Example of non-injectivity}
I. Ciocan-Fontanine has kindly pointed out that, contrary to the case of stable maps, $k$ might not be a closed embedding, even though $i$ is. For instance, consider the Segre embedding
\begin{align*}
\PP^1\times\PP^1 & \xhookrightarrow{i} \PP^3\\ 
([x:y],[z:w]) & \mapsto [xz:xw:yz:yw]\end{align*}
and the induced morphism between quasimap spaces:
\begin{equation*} k\colon \Q{0}{3}{\PP^1\times\PP^1}{(2,2)}\to\Q{0}{3}{\PP^3}{4} \end{equation*}
If we take the following two objects of $\Q{0}{3}{\PP^1 \times \PP^1}{(2,2)}$:
\begin{align*}
  \left(\left(\PP^1_{[s:t]},0,1,\infty\right),\left(L_1=\OO_{\PP^1}(2),u_1=s^2 ,v_1=st\right),\left(L_2=\OO_{\PP^1}(2), u_2=st,v_2=t^2\right)\right)\\
  \left(\left(\PP^1_{[s:t]},0,1,\infty\right),\left(L_1=\OO_{\PP^1}(2),u_1=st ,v_1=t^2 \right),\left(L_2=\OO_{\PP^1}(2), u_2=s^2,v_2=st\right)\right)
\end{align*}
then these two quasimaps are non-isomorphic, but they both map to the same object under $k$, namely:
 \[
   \left(\left(\PP^1_{[s:t]},0,1,\infty\right),\left(L=\OO_{\PP^1}(4),z_0=s^3t,z_1=s^2t^2,z_2=s^2t^2,z_3=st^3\right)\right)
 \]
Notice that this only happens on the locus of quasimaps with basepoints.
\end{remark}

\begin{remark} The above discussion also makes sense for $\epsilon$-stable quasimaps where $\epsilon > 0$ is an arbitrary rational number. We therefore have a notion of \emph{$\epsilon$-stable relative quasimap}. For $\epsilon=0+$ we recover relative quasimaps as above, whereas for $\epsilon>1$ we recover relative stable maps in the sense of Gathmann.

For simplicity we restrict ourselves to the case $\epsilon=0+$. However, all of the arguments can be adapted to the general case. As $\epsilon$ increases, the recursion formula (see \S \ref{Section recursion formula in general case}) becomes progressively more complicated due to the presence of rational tails of lower and lower degree. Consequently the quantum Lefschetz theorem (see \S \ref{Section quasimap mirror theorem}) also becomes more complicated.
\end{remark}

\begin{remark} To avoid confusion, we remark that the spaces defined above are not the same as the spaces defined in \cite[\S 6]{OkounkovKTheory}, which also go by the name of ``relative quasimaps.''
\end{remark}

\subsection{Basic properties of the moduli space} \label{Subsection basic properties of the moduli space}
\begin{lem}
 $\Q{0}{\alpha}{\PP^N|H}{d}$ is irreducible of codimension $\sum_i\alpha_i$ in $\Q{0}{n}{\PP^N}{d}$.
\end{lem}
\begin{proof} Since $\Q{0}{\alpha}{\PP^N|H}{d}$ is defined as the closure of the nice locus, it is enough to show that the nice locus itself is irreducible of the correct dimension. The quasimaps which appear in the nice locus have no basepoints, and as such may be viewed simply as stable maps. This identifies the nice locus with a locally closed substack of $\M{0}{n}{\PP^N}{d}$, and then \cite[Lemma 1.8]{Ga} applies to show that this has the desired properties.\end{proof}

Recall that there exists a comparison morphism
\begin{equation*}\chi : \M{0}{n}{\PP^N}{d} \to \Q{0}{n}{\PP^N}{d} \end{equation*}
which has the effect of contracting each rational tail and introducing a basepoint at the corresponding node, with multiplicity equal to the degree of the rational tail. (For more details, see \cite[Theorem 3]{MOP} and \cite[\S 4.3]{Manolache-Push}; earlier manifestations of these ideas can be found in \cite{Bertram} and \cite{Popa-Roth}.)

\begin{lem}\label{Comparison morphism birational}
 The comparison morphism $\chi$ restricts to a proper and birational morphism \[\chi_\alpha\colon \M{0}{\alpha}{\PP^N|H}{d}\to \Q{0}{\alpha}{\PP^N|H}{d}.\]
\end{lem}
\begin{proof} Let $\chi_\alpha$ denote the restriction of $\chi$ to $\M{0}{\alpha}{\PP^N|H}{d}$. The fact that this factors through $\Q{0}{\alpha}{\PP^N|H}{d}$ can be deduced by applying $\chi_\alpha$ to a smoothing family of any relative stable map. Moreover $\chi_\alpha$ is an isomorphism over the nice locus, which is an open dense subset of both source and target. It follows that $\chi_\alpha$ is proper and birational (and hence also surjective).\end{proof}

Since the moduli space of relative quasimaps is irreducible of the correct dimension, it has a fundamental class which we can use to define \emph{relative quasimap invariants} for the pair $(\PP^N,H)$:
\begin{equation*} \left\langle \gamma_1 \psi_1^{k_1} , \ldots, \gamma_n \psi_n^{k_n} \right\rangle_{0,\alpha,d}^{\PP^N|H} := \int_{[\Q{0}{\alpha}{\PP^N|H}{d}]} \prod_{i=1}^n \ev_i^* \gamma_i \cdot \psi_i^{k_i} \end{equation*}

\begin{remark}
The relative quasimap invariants of $(\PP^N,H)$ agree with the relative Gromov--Witten invariants of $(\PP^N,H)$, since the birational map $\chi_\alpha$ preserves the fundamental classes (the psi classes pull back along $\chi_\alpha$ by Lemma \ref{lem:compare_psi} below). Note that if we set $\alpha=(0,\ldots,0)$ we recover the classical comparison theorem for the absolute Gromov--Witten and quasimap invariants of projective space.
\end{remark}

We will now use the above results to define relative quasimap invariants in general. Since the absolute space $\Q{0}{n}{\PP^N}{d}$ is unobstructed, the morphism $k\colon\Q{0}{n}{X}{\beta}\to \Q{0}{n}{\PP^N}{d}$ admits a natural relative perfect obstruction theory, and so there is a virtual pull-back morphism $k^!_{\operatorname{v}}$. Alternatively, we may use the presence of a virtual class on $\Q{0}{n}{X}{\beta}$ and the smoothness of $\Q{0}{n}{\PP^N}{d}$ to define a diagonal pull-back morphism $k^!_{\Delta}$. The discussion in Appendix \ref{appendix:intersection} shows that these two maps coincide, and from now on we will denote them both by $k^!$. We then define the virtual class on $\Q{0}{\alpha}{X|Y}{\beta}$ by pull-back along $k$
\begin{equation*} \virt{\Q{0}{\alpha}{X|Y}{\beta}} := k^! [ \Q{0}{\alpha}{\PP^N|H}{d} ] \end{equation*}
and use this class to define relative quasimap invariants in general:
\begin{equation*} \left\langle \gamma_1 \psi_1^{k_1} , \ldots, \gamma_n \psi_n^{k_n} \right\rangle_{0,\alpha,\beta}^{X|Y} := \int_{\virt{\Q{0}{\alpha}{X|Y}{\beta}}} \prod_{i=1}^n \ev_i^* \gamma_i \cdot \psi_i^{k_i}. \end{equation*}

It will be important for the statement of the recursion formula to provide a description of the geometric points of $\Q{0}{\alpha}{X|Y}{\beta}$. Recall the notation introduced at the beginning of \S \ref{Subsection relative stable quasimaps}.
\begin{lem}[Combinatorial description] Let
\begin{equation*} \big((C,x_1,\ldots,x_n), (L_\rho,u_\rho)_{\rho \in \Sigma(1)}, (\varphi_m)_{m \in M}\big) \in \Q{0}{n}{X}{\beta}\end{equation*}
be a quasimap to $X$. This belongs to the relative space $\Q{0}{\alpha}{X|Y}{\beta}\subseteq \Q{0}{n}{X}{\beta}$ if and only if, for every connected component $Z$ of $u_Y^{-1}(0) \subseteq C$, the following conditions hold:
\begin{enumerate}
\item[(i)] if $Z$ consists of an isolated marked point $x_i$, then $u_Y^*(0)$ has order at least $\alpha_i$ at $x_i$;
\item[(ii)] if $Z$ is a (possibly reducible) subcurve of $C$, and if we let $C^{(i)}$ for $1 \leq i \leq r$ denote the irreducible components of $C$ adjacent to $Z$, and $m^{(i)}$ the multiplicity of $u_Y|_{C^{(i)}}^*(0)$ at the unique node $Z \cap C^{(i)}$, then:
\begin{equation} \label{Relative quasimap internal component inequality} \deg(L_{Y}|_Z) + \sum_{i=1}^r m^{(i)} \geq \sum_{x_i \in Z} \alpha_i. \end{equation}
\end{enumerate}
\end{lem}
\begin{proof}
From the definition of $\Q{0}{\alpha}{X|Y}{\beta}$, we see that it is sufficient to prove the statement in the case $(X,Y)=(\PP^N,H)$. These conditions are satisfied on the nice locus by definition, and so continue to be satisfied on the closure $\Q{0}{\alpha}{\PP^N|H}{d}$ by the conservation of number principle. On the other hand, suppose we are given a quasimap satisfying these conditions. If $x$ is a basepoint of this quasimap, of multiplicity $m$, then we may adjoin a rational tail to the source curve at $x$, and define a map from this rational tail to $\PP^N$ by choosing a line in $\PP^N$ through the image of $x$ and taking an $m$-fold cover of this line, totally ramified at the point of intersection of the rational tail with the rest of the curve. In this way we obtain a stable map which maps down to our original quasimap under $\chi$. Moreover this stable map belongs to $\M{0}{\alpha}{\PP^N|H}{d}$ by \cite[Proposition 1.14]{Ga}; hence by Lemma \ref{Comparison morphism birational} our quasimap belongs to $\Q{0}{\alpha}{\PP^N|H}{d}$, as required.
\end{proof}

\section{Recursion formula}\label{section recursion}\label{section recursion formula}
\subsection{Recursion formula for $(\PP^N,H)$} \label{Section recursion for PN}
\label{Subsection recursion formula for PN}
We wish to obtain a recursion formula relating the quasimap invariants of multiplicity $\alpha$ with the quasimap invariants of multiplicity $\alpha + e_k$, as in \cite[Theorem 2.6]{Ga}. For $m = \alpha_k + 1$ the following section (of the pull-back of the jet bundle of the universal line bundle)
\[
\sigma^m_k := x_k^*d^m_{\mathcal C/\om{Q}}(u_0)\in \HH^0(\om{Q},x_k^*\mathcal P^m_{\mathcal C/\mathcal Q}(\mathcal L))
\]
vanishes along $\Q{0}{\alpha+e_k}{\PP^N|H}{d}$ inside $\om{Q} = \Q{0}{\alpha}{\PP^N|H}{d}$, and also along a number of \ilemph{comb loci}.  The latter parametrise quasimaps for which $x_k$ belongs to an internal component $Z \subseteq C$ (a connected component of the vanishing locus of $u_0$), such that:
\begin{equation*}\deg(L|_{Z})+\sum_{i=1}^r m^{(i)}=\sum_{x_i\in Z}\alpha_i \end{equation*}
The strong stability condition means that quasimaps in the comb loci cannot contain any rational tails; this is really the only difference with the case of stable maps.

Indeed, we can push forward Gathmann's recursion formula for stable maps along the comparison morphism
\begin{equation*} \comp \colon \M{0}{\alpha}{\PP^N|H}{d}\to\Q{0}{\alpha}{\PP^N|H}{d} \end{equation*}
and, due to Corollary \ref{Comparison morphism birational} above, the only terms which change are the comb loci containing rational tails. In fact these disappear, since the restriction of the comparison map to these loci has positive-dimensional fibres:
\begin{lem}\label{lem:posdimfiber} Consider a rational tail component in the comb locus of the moduli space of stable maps, i.e. a moduli space of the form:
\begin{equation*} \M{0}{(m^{(i)})}{\PP^N|H}{d} \end{equation*}
and assume that $Nd>1$.  Then
\begin{equation*} \dim \left( [\M{0}{(m^{(i)})}{\PP^N|H}{d}] \cap \ev_1^*(\pt_H) \right) > 0 \end{equation*}
where $\pt_H \in \Achow^{N-1}(H)$ is a point class. Thus the pushforward along $\comp$ of any comb locus with a rational tail is zero.
\end{lem}
\begin{proof} This is a simple dimension count. We have
\begin{align*} \dim \left( [\M{0}{(m^{(i)})}{\PP^N|H}{d}] \cap \ev_1^*(\pt_H) \right) & =(N-3)+d(N+1)+(1-m^{(i)})-(N-1) \\
& =(Nd-1)+(d-m^{(i)})
\end{align*}
from which the lemma follows because $m^{(i)} \leq d$.
\end{proof}
\begin{remark} With an eye to the future, we remark that these rational tail components contribute nontrivially to the Gromov--Witten invariants of a Calabi--Yau hypersurface in projective space, and so their absence from the quasimap recursion formula accounts for the divergence between Gromov--Witten and quasimap invariants in the Calabi--Yau case \cite[Rmk. 1.6]{Ga-MF}. \end{remark}

Since we wish to apply the projection formula to Gathmann's recursion relation, we should express the cohomological terms which appear as pull-backs:
\begin{lem}\label{lem:compare_psi} We have:
\begin{align*} \comp^*(\psi_k) & =\psi_k \\
\comp^*(\ev_k^* H)& =\ev_k^* H
\end{align*}
\end{lem}
\begin{proof}
The contraction of rational tails occurs away from the markings.
\end{proof}
\begin{prop} \label{Recursion formula for PN}
Define the \ilemph{quasimap comb locus} $\mathcal{D}^\mathcal{Q}_{\alpha,k}(\PP^N|H,d)$ as the union of the moduli spaces
\begin{equation*}
\mathcal{D}^{\mathcal{Q}}(\PP^N|H,A,B,M) := \Q{0}{A^{(0)} \cup \{q_1^0,\ldots,q_r^0\}}{H}{d_0}\times_{H^r}\prod_{i=1}^r \Q{0}{\alpha^{(i)}\cup (m^{(i)})}{\PP^N|H}{d_i}
\end{equation*}
where the union runs over all splittings $A =(A^{(0)},\ldots,A^{(r)})$ of the markings (inducing a splitting $(\alpha^{(0)}, \ldots, \alpha^{(r)})$ of the corresponding tangency conditions), $B = (d_0, \ldots, d_r)$ of the degree and all valid multiplicities $M = (m^{(1)}, \ldots, m^{(r)})$ such that the above spaces are all well-defined (in particular we require that $|A^{(0)}|+r$ and $|A^{(i)}|+1$ are all $\geq 2$) and such that

\[
d_0+\sum_{i=1}^r m^{(i)}=\sum \alpha^{(0)}
\]
Write $[\mathcal{D}^\mathcal{Q}_{\alpha,k}(\PP^N|H,d)]$ for the sum of the (product) fundamental classes, where each term is weighted by:
\begin{equation*}\dfrac{m^{(1)} \cdots m^{(r)}}{r!} \end{equation*}
Then
\[
(\alpha_k\psi_k+\ev_k^*H)\cdot[\Q{0}{\alpha}{\PP^N|H}{d}]=[\Q{0}{\alpha+e_k}{\PP^N|H}{d}]+[\mathcal{D}^\mathcal{Q}_{\alpha,k}(\PP^N|H,d)].
\]
\end{prop}
\begin{proof}
This follows from \cite[Thm. 2.6]{Ga} by pushing forward along $\chi$, using the projection formula and Lemmas \ref{Comparison morphism birational}, \ref{lem:posdimfiber} and \ref{lem:compare_psi} .
\end{proof}

\begin{remark} \label{Remark on definition of comb locus} In the discussion above we have implicitly used the fact that there exists a commuting diagram of comb loci:
\bcd
\mathcal{D}^{\mathcal{M}}(\PP^N|H,A,B,M) \ar[r] \ar[d] & \M{0}{\alpha}{\PP^N|H}{d} \ar[d] \\
\mathcal{D}^{\mathcal{Q}}(\PP^N|H,A,B,M) \ar[r] & \Q{0}{\alpha}{\PP^N|H}{d}
\ecd
The vertical arrow on the left is a product of comparison morphisms (notice that $H\cong\PP^{N-1}$). The horizontal arrow at the top is the gluing morphism which glues together the various pieces of the comb to produce a single relative stable map. Here we explain how to define the corresponding gluing morphism for quasimaps, that is, the bottom horizontal arrow.

Suppose for simplicity that $r$, the number of teeth of the comb, is equal to $1$. Consider an element of the quasimap comb locus, consisting of two quasimaps:
\begin{align*} & \big((C^0,x^0_1,\ldots,x^0_{n_0},q^0),L^0,u^0_0, \ldots, u^0_N\big) \\
&\big((C^1,x^1_1,\ldots,x^1_{n_1},q^1),L^1,u^1_0, \ldots, u^1_N\big) \end{align*}
such that $u^0(q^0) = u^1(q^1)$ in $\PP^N$. We want to glue these quasimaps together at $q^0$,~$q^1$. The definition of the curve is obvious; we simply take:
\begin{equation*} C = C^0 \,_{q^0}\!\sqcup_{q^1} C^1 \end{equation*}
On the other hand, gluing together the line bundles $L^0$ and $L^1$ to obtain a line bundle $L$ over $C$ requires a choice of scalar $\lambda \in \Gm$, in order to match up the fibres over $q^i$. Furthermore, if the sections are to extend as well, then this scalar must be chosen in such a way that it takes $(u^0_0(q^0), \ldots, u^0_N(q^0)) \in (L^0_{q^0})^{\oplus (N+1)}$ to $(u^1_0(q^1), \ldots, u^1_N(q^1)) \in (L^1_{q^1})^{\oplus (N+1)}$. Since neither $q^0$ nor $q^1$ are basepoints (because they are markings), these tuples are nonzero, and so $\lambda$ is unique if it exists. Furthermore it exists if and only if these tuples belong to the same $\Gm$-orbit in $\Aaff^{N+1}$. This is precisely the statement that $u^0(q^0) = u^1(q^1) \in \PP^N$. Similar arguments apply for $r>1$, and for more general toric varieties. \end{remark}

\subsection{Recursion formula in the general case}\label{Section recursion formula in general case}

In this section we prove the main result of this paper: a recursion formula for relative quasimap invariants of a general pair $(X,Y)$.  

\begin{thm} \label{Theorem general recursion} Let $X$ be a smooth projective toric variety and let $Y \subseteq X$ be a very ample hypersurface (not necessarily toric). Then
\begin{equation*} (\alpha_k \psi_k + \ev_k^* [Y]) \cap \virt{\Q{0}{\alpha}{X|Y}{\beta}} = \virt{\Q{0}{\alpha+e_k}{X|Y}{\beta}} + \virt{\mathcal D^\mathcal{Q}_{\alpha,k}(X|Y,\beta)} \end{equation*}
in the Chow group of $\Q{0}{\alpha}{X|Y}{\beta}$. 
\end{thm}

The formula is proven by pulling back the recursion for $(\PP^N,H)$ along $k=\om{Q}(i)$. Only the final term requires further discussion. As in the previous section, we define $\mathcal D^\mathcal{Q}_{\alpha,k}(X|Y,\beta)$ to be the union of the moduli spaces
\begin{equation*} \mathcal D^{\mathcal{Q}}(X|Y,A,B,M) := \Q{0}{A^{(0)} \cup \{q_1, \ldots, q_r\}}{Y}{\beta^{(0)}} \times_{Y^r} \prod_{i=1}^r \Q{0}{\alpha^{(i)}\cup (m_i)}{X|Y}{\beta^{(i)}} \end{equation*}
where the union runs over all splittings $A = (A^{(0)},\ldots,A^{(r)})$ of the markings (inducing a splitting $(\alpha^{(0)}, \ldots, \alpha^{(r)})$ of the corresponding tangency requirements), $B = (\beta^{(0)}, \ldots, \beta^{(r)})$ of the curve class $\beta$ and all valid multiplicities $M = (m^{(1)}, \ldots, m^{(r)})$ such that the above spaces are non-empty and such that:
\[
Y \cdot \beta^{(0)} +\sum_{i=1}^r m^{(i)}=\sum \alpha^{(0)}
\]
We refer to the $\mathcal D^{\mathcal{Q}}(X|Y,A,B,M)$ as \emph{comb loci}.

\begin{remark}\label{Section comparison with GIT construction}
\label{GIT comparison remark}Note that $Y$ is not in general toric, and so we should clarify the meaning of the factor
\begin{equation*} \Q{0}{A^{(0)} \cup \{ q_1, \ldots, q_n \}}{Y}{\beta^{(0)}} \end{equation*}
above. There are two possibilities here: one is to \emph{define} this space as the cartesian product
\bcd
\om{Q}_{0,n}(Y,\beta) \ar[r] \ar[d] \ar[rd,phantom,"\square"] & \om{Q}_{0,n}(H,d) \ar[d] \\
\om{Q}_{0,n}(X,\beta) \ar[r,"k"] & \om{Q}_{0,n}(\PP^N,d)
\ecd
and equip it with the virtual class pulled back along $k$:
\begin{equation*} \virt{\om{Q}_{0,n}(Y,\beta)} := k^! [ \om{Q}_{0,n}(H,d) ] \end{equation*}
Using this definition, $\om{Q}_{0,n}(Y,\beta)$ consists of those quasimaps in $\om{Q}_{0,n}(X,\beta)$ for which the section $u_Y$ (constructed in \S \ref{Subsection relative stable quasimaps}) is identically zero.
This has obvious advantages from the point of view of our computations, but is conceptually unsatisfying. 

On the other hand, in \cite{CFKM} moduli spaces of stable quasimaps are constructed for GIT quotient targets satisfying a number of conditions. Since $Y$ is a hypersurface in a toric variety, it has a natural presentation as such a GIT quotient
\begin{equation*} Y = C(Y) \sslash G \end{equation*}
where $C(Y)\subseteq \mathbb A^{\Sigma_X(1)}$ is the affine cone over $Y$ and $G=\Hom_{\mathbb Z}(\Pic(X),\Gm)\cong\Gm^{r_X}$ acts on $C(Y)$ via the natural inclusion
\begin{equation*} \Gm^{r_X}\hookrightarrow \Gm^{\Sigma_X(1)} \end{equation*}
(here $C(Y)\subseteq \mathbb{A}^{\Sigma_X(1)}$ is preserved by $G$ because it is cut out by a homogeneous polynomial in the Cox ring of $X$). Thus, we have two possible definitions of $\om{Q}_{0,n}(Y,\beta)$ and its virtual class; we will now show that they agree.

Objects of $\om{Q}^{\operatorname{GIT}}_{0,n}(Y,\beta)$ are diagrams of the form
\bcd
P\times_{G} C(Y) \ar[d,"p" left] \\
C \ar[u,bend right,"u"right]
\ecd
where $C$ is a prestable curve, $P$ is a principal $G$-bundle on $C$, and $u$ is a section of the associated $C(Y)$-bundle. Given this data, there is a $G$-equivariant embedding
\bcd
P\times_{G} C(Y) \ar[r,hook, "j"]\ar[d,"p" left] & P\times_{G}\mathbb A^{\Sigma_X(1)}= \bigoplus_{\rho\in\Sigma_X(1)} L_{\rho}\ar[dl]\\
C \ar[u,bend right,"u"right] &
\ecd
which expresses $P \times_G C(Y)$ as the vanishing locus of $u_Y$, viewed as a section of a line bundle on the total space of $\oplus_{\rho \in \Sigma_X(1)} L_\rho$. This shows that the two definitions of the moduli space agree.

It remains to compare the virtual classes. The obstruction theory on the GIT space is defined relative to the stack $\mathfrak{Bun}^{G}_{0,n}$ parametrising principal $G$-bundles on the universal curve
$\mathcal{C}_{\MM_{0,n}} \to \MM_{0,n}$.
It is given by
\begin{equation*} \EE_{\om{Q}/\mathfrak{Bun}^G_{0,n}}^\vee = \R \pi_* (u^*\TT_p) \end{equation*}
where $\pi$ is the universal curve over $\om{Q} = \om{Q}^{\operatorname{GIT}}_{0,n}(Y,\beta)$ and $\TT_p$ is the relative tangent complex of the projection map $\rho$. There is a natural isomorphism
\begin{equation*} \mathfrak{Bun}^{G}_{0,n} = \underbrace{\mathfrak{Pic}_{0,n} \times_{\MM_{0,n}} \ldots \times_{\MM_{0,n}} \mathfrak{Pic}_{0,n}}_{r_X} \end{equation*}
given by sending $P$ to the $r_X$ individual factors of the affine bundle $P\times_{G}\mathbb A^{r_X}$. 
Using the normal sheaf sequence for the inclusion $j$ in the diagram above (all relative to the base $C$) we obtain a short exact sequence on $C$:
\begin{equation*} 0 \to u^* \TT_p \to \bigoplus_{\rho \in \Sigma_X(1)} L_\rho \to u^* \NN_{P \times_G C(Y)/\oplus_{\rho \in \Sigma_X(1)} L_\rho} \to 0 \end{equation*}
Since $P \times_G C(Y)$ is defined by the vanishing of $u_Y$, we see that the final term is isomorphic to the line bundle $L_Y$ discussed above. Thus we have a natural isomorphism of objects of the derived category:
\begin{equation*} u^* \TT_p = \left[ \bigoplus_{\rho \in \Sigma_X(1)} L_\rho \to L_Y \right] \end{equation*}
Applying $\R \pi_*$ we obtain on the left hand side the obstruction theory for $\om{Q}^{\operatorname{GIT}}_{0,n}(Y,\beta)$ relative $\mathfrak{Bun}^G_{0,n}$. On the other hand, the first term on the right hand side is the obstruction theory for the toric quasimap space $\om{Q}_{0,n}(X,\beta)$ relative to the fibre product of the Picard stacks, 
whereas the second term is the relative obstruction theory for $\om{Q}_{0,n}(Y,\beta)$ inside $\om{Q}_{0,n}(X,\beta)$. Thus the virtual classes agree as well.

\end{remark}

\begin{aside} In Remark \ref{Example of non-injectivity} we saw that if $Y=\PP^1 \times \PP^1$ and $X=\PP^3$, with $Y \hookrightarrow X$ given by the Segre embedding, then the induced map
\begin{equation*} \mathcal{Q}^{\operatorname{GIT}}_{0,3}(Y,(2,2)) \to \mathcal{Q}^{\operatorname{GIT}}_{0,3}(X,4)\end{equation*}
is not injective. However, there is no contradiction between this and the discussion above. The somewhat subtle point is that the definition of the quasimap space depends on the presentation of the target as a GIT quotient \cite[\S 4.6]{CFKM}. In Remark \ref{Example of non-injectivity} we expressed $Y$ as a toric GIT quotient
\begin{equation*} Y \cong \Aaff^4 \sslash \Gm^2 \end{equation*}
whereas in the context of Remark \ref{Section comparison with GIT construction}, $Y$ would be expressed as a more parsimonious quotient:
\begin{equation*} Y \cong C(Y) \sslash \Gm \end{equation*}
The map $\mathcal{Q}^{\operatorname{GIT}}(\Aaff^4\sslash\Gm^2) \to \mathcal{Q}^{\operatorname{GIT}}(X)$ is not an embedding, but it factors through $\mathcal{Q}^{\operatorname{GIT}}(C(Y)\sslash\Gm) \rightarrow \mathcal{Q}^{\operatorname{GIT}}(X)$ which is.

\end{aside}

We have thus shown that the comb locus $\mathcal D^{\mathcal{Q}}(X|Y,A,B,M)$ makes sense for general $(X,Y)$. Our next task is to construct a virtual class on this locus.  Consider the product (\emph{not} the fibre product over $Y^r$)
\begin{equation*} \mathcal E^{\mathcal{Q}}(X|Y,A,B,M) := \Q{0}{A^{(0)} \cup \{q_1, \ldots, q_r\}}{Y}{\beta^{(0)}} \times \prod_{i=1}^r \Q{0}{\alpha^{(i)}\cup (m_i)}{X|Y}{\beta^{(i)}} \end{equation*}
which we may endow with the product virtual class (with weighting as before):
\begin{align*} \virt{\mathcal E^{\mathcal{Q}}(X|Y, & A,B,M)} := \\
& \left( \dfrac{m^{(1)} \cdots m^{(r)}}{r!}\right) \cdot \left( \virt{\Q{0}{A^{(0)} \cup \{q_1, \ldots, q_r\}}{Y}{\beta^{(0)}}} \times \prod_{i=1}^r \virt{\Q{0}{\alpha^{(i)}\cup (m_i)}{X|Y}{\beta^{(i)}}} \right) \end{align*}
We have the following cartesian diagram
\bcd
\mathcal D^{\mathcal{Q}}(X|Y,A,B,M)\ar[r]\ar[d]\ar[dr,phantom,"\Box"] & \mathcal E^{\mathcal{Q}}(X|Y,A,B,M)\ar[d] \\
X^r\ar[r,"\Delta_{X^r}"] & X^r\times X^r
\ecd
and we can use this to define a virtual class on the comb locus:
\[
 \virt{\mathcal D^{\mathcal{Q}}(X|Y,A,B,M)} := \Delta_{X^r}^!\virt{\mathcal E^{\mathcal{Q}}(X|Y,A,B,M)}
\]
The virtual class on the union $\mathcal D^\mathcal{Q}_{\alpha,k}(X|Y,\beta)$ of the comb loci is defined to be the sum of the virtual classes $\virt{\mathcal D^{\mathcal{Q}}(X|Y,A,B,M)}$.

\begin{remark} This is the same definition of the virtual class of the comb locus that we gave in \S \ref{Subsection recursion formula for PN} in the case $(X,Y)=(\PP^N,H)$. \end{remark}

On the other hand, there is another cartesian diagram:
\bcd
{\displaystyle \coprod_{B \colon i_* B = B^\prime} \mathcal D^\mathcal{Q}(X|Y,A,B,M)} \ar[r] \ar[d] & \mathcal D^\mathcal{Q}(\PP^N|H, A, B^\prime, M) \ar[d] \ar[ld,phantom,"\square"]  \\
\Q{0}{n}{X}{\beta} \ar[r,"k"] & \Q{0}{n}{\PP^N}{d}
\ecd
Recall that we are trying to show that the virtual class of the comb locus pulls back nicely along $k$. The result that we need is:
\begin{lem} \label{Comb loci pull back} $\displaystyle k^! \virt{\mathcal D^\mathcal{Q}(\PP^N|H,A,B^\prime,M)} = \sum_{B : i_* B = B^\prime} \virt{\mathcal D^\mathcal{Q}(X|Y,A,B,M)}$ \end{lem}

For the proof of Lemma~\ref{Comb loci pull back}, let us introduce the following shorthand notation.  We fix the data of $A$, $B^\prime$, $M$ for a comb locus of $(\PP^N,H)$, and set:
\begin{align*}
\mathcal{D}(X|Y) & := \textstyle \coprod_{B \colon i_* B = B^\prime} \mathcal D^{\mathcal{Q}}(X|Y,A,B,M) &&& \mathcal{D}(\PP^N|H) & := \mathcal D^{\mathcal{Q}}(\PP^N|H,A,B^\prime,M)\\
\mathcal{E}(X|Y) & := \textstyle \coprod_{B \colon i_* B = B^\prime} \mathcal E^{\mathcal{Q}}(X|Y,A,B,M) &&& \mathcal{E}(\PP^N|H) & := \mathcal E^{\mathcal{Q}}(\PP^N|H,A,B^\prime,M)\\
\mathcal{D}(X) & := \textstyle \coprod_{B \colon i_* B = B^\prime} \mathcal{D}^{\mathcal{Q}}(X,A,B)  &&& \mathcal{D}(\PP^N) & := \mathcal D^{\mathcal{Q}}(\PP^N,A,B^\prime)\\
\mathcal{E}(X) & := \textstyle \coprod_{B \colon i_* B = B^\prime} \mathcal{E}^{\mathcal{Q}}(X,A,B) &&& \mathcal{E}(\PP^N) & := \mathcal E^{\mathcal{Q}}(\PP^N,A,B^\prime)\\
\om{Q}(X) & :=\Q{0}{n}{X}{\beta} &&& \om{Q}(\PP^N) & :=\Q{0}{n}{\PP^N}{i_* \beta} 
\end{align*}
Here $\mathcal{D}(X)$ and $\mathcal{E}(X)$ are the so-called centipede loci; they are defined in the same way as the comb loci, except that we replace both the quasimaps to $Y$ and the relative quasimaps to $(X,Y)$ by quasimaps to $X$. There is a cartesian diagram:
\bcd
\mathcal{E}(X|Y) \ar[d]\ar[r]\ar[dr,phantom,"\Box"] & \mathcal{E}(\PP^N|H)\ar[d,"\theta"] \\
\mathcal{E}(X)\ar[r] & \mathcal{E}(\PP^N)
\ecd
Since $\mathcal{E}(\PP^N)$ is smooth (being a product of spaces of quasimaps to $\PP^N$) and there is a natural fundamental class on $\mathcal{E}(\PP^N|H)$, we have a diagonal pull-back morphism $\theta^! = \theta_{\Delta}^!$ (see Appendix~\ref{appendix:intersection}). It follows immediately from the definitions that:
\begin{lem}\label{theta-pull} $\virt{\mathcal{E}(X|Y)}=\theta^!\virt{\mathcal{E}(X)}$ .\end{lem}

Now consider the following cartesian diagram
\bcd
\mathcal D(X)\ar[r]\ar[d,"\varphi_{{}_X}"]\ar[dr,phantom,"\Box"] & \mathcal D(\PP^N)\ar[r]\ar[d,"\varphi_{\PP^N}"]\ar[dr,phantom,"\Box"] & \MM_{A,B}^{\operatorname{wt}}\ar[d,"\psi"] \\
\om{Q}(X)\ar[r,"k"] & \om{Q}(\PP^N)\ar[r] & \MM_{0,n,\beta}^{\operatorname{wt}}
\ecd
where $\MM_{0,n,\beta}^{\operatorname{wt}}$ is the moduli space of prestable curves weighted by the class~$\beta$ \cite[\S 2]{Costello} and:
\begin{equation*} \MM_{A,B}^{\operatorname{wt}} := \MM_{0,A^{(0)} \cup \{ q_1^0, \ldots, q_r^0 \}, \beta^{(0)}}^{\operatorname{wt}} \times \prod_{i=1}^r \MM_{0,A^{(i)} \cup \{ q_i^1 \},\beta^{(i)}}^{\operatorname{wt}} \end{equation*}
The vertical maps in the above diagram are given by gluing together curves (in the case of $\psi$) and quasimaps (in the case of $\varphi_X$ and $\varphi_{\PP^N}$). The maps $\mathcal{D}(X) \to \MM_{A,B}^{\operatorname{wt}}$ and $\om{Q}(X) \to \MM_{0,n,\beta}^{\operatorname{wt}}$ admit relative perfect obstruction theories which are the same as the usual perfect obstruction theories relative to the moduli spaces of \emph{unweighted} curves. Furthermore the morphism $\psi$ admits a perfect obstruction theory. Thus there are virtual pull-back morphisms $\psi^!$, and by the splitting axiom (which is the same in quasimap theory as in Gromov--Witten theory; see \cite[\S 2.3.3]{CF-K-higher-genus}) we have:
\begin{equation*} \virt{\mathcal{D}(X)} := \Delta_{X^r}^! \virt{\mathcal{E}(X)} = \psi^! \virt{\om{Q}(X)} \end{equation*}
Commutativity of virtual pull-backs then implies that:
\begin{equation} \virt{\mathcal{D}(X)} = \label{psishriek formula} \psi^!\virt{\om{Q}(X)}= \psi^! k^! [ \om{Q}(\PP^N)] = k^!\psi^![\om{Q}(\PP^N)] = k^! [ \mathcal{D}(\PP^N) ]\end{equation}

\begin{proof}[Proof of Lemma \ref{Comb loci pull back}] Putting all the preceding results together, we consider the cartesian diagram:
\bcd
\mathcal D(X|Y)\ar[d]\ar[r]\ar[dr,phantom,"\Box"] & \mathcal E(X|Y)\ar[d]\ar[r]\ar[dr,phantom,"\Box"] & \mathcal E(\PP^N|H)\ar[d,"\theta"] \\
\mathcal D(X)\ar[d]\ar[r]\ar[dr,phantom,"\Box"] & \mathcal E(X)\ar[d]\ar[r] & \mathcal E(\PP^N) \\
X^r\ar[r,"\Delta_{X^r}"] & X^r\times X^r & {}
\ecd
We then have:
\begin{align*} \virt{\mathcal D(X|Y)} & = \Delta_{X^r}^! \virt{\mathcal E(X|Y)} & \text{by definition}\\
& =  \Delta_{X^r}^! \theta^!\virt{\mathcal E(X)} & \text{by Lemma \ref{theta-pull}} \\
& =  \theta^!\Delta_{X^r}^! \virt{\mathcal E(X)} & \text{by commutativity} \\
& =  \theta^! \virt{\mathcal{D}(X)} & \text{by definition} \\
& =  \theta^!k^! [\mathcal{D}(\PP^N)] & \text{by formula \eqref{psishriek formula} above} \\
& =  \theta^!k^!\Delta_{(\PP^N)^r}^! [\mathcal E(\PP^N)] & \text{by definition} \\
& =  k^!\Delta_{(\PP^N)^r}^!\theta^! [\mathcal E(\PP^N)] & \text{by commutativity} \\
& = k^! \Delta_{(\PP^N)^r}^! [\mathcal{E}(\PP^N|H)] & \text{by Lemma \ref{theta-pull}} \\
& =  k^![\mathcal{D}(\PP^N|H)] & \text{by definition}
\end{align*}
Summing over all the components of $\mathcal{D}^{\mathcal{Q}}_{\alpha,k}(\PP^N|H,d)$ we obtain the result. \end{proof}

\begin{proof}[Proof of Theorem~\ref{Theorem general recursion}]
Apply $k^!$ to Proposition~\ref{Recursion formula for PN}, using  Lemma~\ref{Comb loci pull back}. 
\end{proof}

\section{Quasimap quantum Lefschetz theorem} \label{Section quasimap mirror theorem}
The recursion formula shows that the relative quasimap invariants of $(X,Y)$ are completely determined, in an algorithmic way, from the absolute invariants of $X$ and $Y$; by repeatedly applying the recursion formula, we can remove all the tangency conditions, leaving us with an expression which only involves the invariants of $X$ and $Y$.

However, we can do much more than this. In this section we will prove (two variations of) a \emph{quantum Lefschetz theorem for quasimap invariants}, that is, a result which expresses the quasimap invariants of $Y$ in terms of those of $X$. This is the quasimap analogue of the quantum Lefschetz hyperplane principle in Gromov--Witten theory and, on the face of it, has nothing to do with relative invariants.

\subsection{General quasimap quantum Lefschetz}\label{Subsection general quasimap Lefschetz}
First we state the most general form of the theorem, without any additional assumptions on $X$ and $Y$.
\begin{thm}[Quasimap quantum Lefschetz theorem] \label{Theorem full quasimap Lefschetz} Let $X$ be a smooth projective toric variety and $Y \subseteq X$ a smooth very ample hypersurface. Then there is an explicit algorithm to recover the (restricted) absolute quasimap invariants of $Y$, as well as the relative quasimap invariants of $(X,Y)$, from the absolute quasimap invariants of $X$.\end{thm}
The corresponding result in Gromov--Witten theory is due to Gathmann \cite[Corollary 2.5.6]{GathmannThesis}; the proof we present in the quasimap setting is very similar to his. The term \emph{``restricted''} here means that we only integrate against cohomology classes pulled back from $\HH^*(X)$, rather than allowing arbitrary classes from $\HH^*(Y)$.

\begin{proof} The idea, of course, is to repeatedly apply the recursion formula. The proof is by induction, and in order for the argument to work it is essential that we determine simultaneously  the absolute invariants of $Y$ and the relative invariants of $(X,Y)$.

We induct on: the intersection number $d=Y\cdot\beta$, the number of marked points $n$, and the total tangency $\Sigma_i \alpha_i$, in that order. This means that when we come to compute an absolute or relative invariant, we assume that all of the absolute \emph{and} relative invariants with
\begin{enumerate}
\item[(i)] smaller $d$, or
\item[(ii)] the same $d$, but smaller $n$, or
\item[(iii)] the same $d$, the same $n$, but smaller $\Sigma_i \alpha_i$
\end{enumerate}
are known. For the purposes of this ordering, we set $\Sigma_i \alpha_i = d+1$ for any absolute invariant of $Y$. This means that when we come to compute such an invariant, we assume that all the relative invariants with the same $d$ and $n$ are known.

We first prove the induction step for the relative invariants; suppose then that we want to compute some invariant:
\begin{equation*} \langle \gamma_1 \psi_1^{k_1}, \ldots, \gamma_n \psi_n^{k_n} \rangle_{0,\alpha,\beta}^{X|Y}  \end{equation*}
We assume $\Sigma_i \alpha_i > 0$, since otherwise this is just an absolute invariant of $X$. Pick some $k \in \{1,\ldots,n\}$ with $\alpha_k > 0$, and apply Theorem \ref{Theorem general recursion} to obtain:
\begin{equation*} ((\alpha_k-1) \psi_k + \ev_k^* [Y]) \cap \virt{\Q{0}{\alpha-e_k}{X|Y}{\beta}} = \virt{\Q{0}{\alpha}{X|Y}{\beta}} + \virt{\mathcal D^\mathcal{Q}_{\alpha-e_k,k}(X|Y,\beta)} \end{equation*}
Capping this with the appropriate product of evaluation and psi classes, we obtain from the first term on the right-hand side the invariant that we are looking for.

It remains to show that the other terms are known by the induction hypothesis. Clearly, this is true for the term on the left-hand side, which has the same $d$, the same $n$, but smaller $\Sigma_i \alpha_i$. Consider on the other hand a component of the comb locus. This contributes a product of an absolute invariant of $Y$ (corresponding to the internal component) with a number of relative invariants of $(X,Y)$ (corresponding to the external components). One can check that each of these invariants either has smaller $d$, or the same $d$ and smaller $n$. Thus, they are also determined. Therefore the relative invariant is determined inductively.

Now we prove the induction step for the absolute invariants of $Y$. Suppose then that we want to compute a restricted invariant:
\begin{equation*} \langle \gamma_1 \psi_1^{k_1}, \ldots, \gamma_n \psi_n^{k_n} \rangle_{0,n,\beta}^{Y} \end{equation*}
 If we apply Theorem \ref{Theorem general recursion} with $\alpha=(d+1,0,\ldots,0)$ we obtain
\begin{equation*} (d \psi_1 + \ev_1^* [Y]) \cap \virt{\Q{0}{\alpha-e_1}{X|Y}{\beta}} = \virt{\mathcal D^\mathcal{Q}_{\alpha,1}(X|Y,\beta)} \end{equation*}
where the comb locus on the right-hand side has a connected component isomorphic to the moduli space
\begin{equation*} \Q{0}{n}{Y}{\beta} \end{equation*}
(corresponding to a ``comb with no teeth''). Capping as before with an appropriate class, we obtain the invariant that we are looking for. The term on the left-hand side is known since $\Sigma_i \alpha_i$ is smaller, while any other terms coming from the comb locus either involve invariants with smaller $d$ or with the same $d$ but smaller $n$, and so are also known inductively. This completes the proof. \end{proof}

\begin{remark} There is a subtle but extremely important point which we have ignored in the proof above. While the statement of Theorem \ref{Theorem full quasimap Lefschetz} only concerns the \emph{restricted} quasimap invariants, i.e. those with insertions from $\HH^*(X)$, when we calculate contributions from the comb loci we are forced to consider unrestricted invariants, due to classes in the diagonal in $\HH^*(Y\times Y)$ which do not come from $\HH^*(X\times X)$. This is problematic, since in general these terms cannot be computed inductively.

However, a careful analysis of the recursion formula shows that any term which appears in this way must in fact be zero. The argument is the same as the one given for Gromov--Witten invariants in \cite[\S 2.5]{GathmannThesis}; the details are left to the reader. The key idea is to show that any absolute or relative quasimap invariant which has precisely one insertion from outside of $\HH^*(X)$ must be zero, and then to show that any term arising from the comb locus and involving unrestricted classes is equal to a product of invariants, at least one of which takes this form.\end{remark}

\subsection{A mirror theorem for quasimap invariants}
Although the algorithm presented in the previous section is completely explicit, it is in general quite involved, since the combinatorics can become arbitrarily complicated. We would like to be able to find a closed formula which expresses the quasimap invariants of $Y$ in terms of those of $X$. This is our goal over the next few sections, culminating in Theorem \ref{Theorem Quantum Lefschetz}, which provides such a closed formula, under some additional restrictions.

In \cite{Ga-MF} Gathmann applies the stable map recursion formula to obtain a new proof of the mirror theorem for hypersurfaces \cite{Givental-equivariantGW}. This can be viewed as a partial quantum Lefschetz formula, expressing certain stable map invariants of $Y$ in terms of those of $X$.

In this section we carry out a similar computation in the quasimap setting. We work with generating functions for $2$-pointed quasimap invariants (the minimal number of markings, due to the strong stability condition). The absence of rational tails in the quasimap moduli space makes the quasimap recursion much simpler than Gathmann's.

Our formula can be viewed as a special case of \cite[Corollary 5.5.1]{CF-K-wallcrossing}, and thus as a relation between certain residues of the $\Gm$-action on spaces of $0$-pointed and $1$-pointed parametrised quasimaps to $Y$. Some of the consequences of this formula are explored in \cite[Section 5.5]{CF-K-wallcrossing}; for instance, it follows in the semipositive case that all primary $\epsilon$-quasimap invariants with a fundamental class insertion can be expressed in terms of $2$-pointed invariants.

\subsection{Setup} \label{Subsection setup}
As before, we let $X=X_{\Sigma}$ be a smooth projective toric variety and $i \colon Y \hookrightarrow X$ a smooth very ample hypersurface. We also make the following two assumptions:
\begin{enumerate}
\item[(i)] \emph{$Y$ is semi-positive}: $-K_Y$ is nef;
\item[(ii)] \emph{$Y$ contains all curve classes}: the map $i_* : \Achow_1(Y) \to \Achow_1(X)$ is surjective.
\end{enumerate}
By adjunction, $-K_X$ pairs strictly positively with every curve class coming from $Y$, hence with every curve class by Assumption (2). Thus $-K_X$ is ample by Kleiman's criterion (since the effective cone of a toric variety is finitely generated), so $X$ is Fano. Also note that if $\dim X \geq 3$ then Assumption (2) always holds, due to the classical Lefschetz hyperplane theorem; on the other hand if $\dim X = 2$ then Assumption (2) forces $X$ to be $\PP^2$.

We fix a homogeneous basis $\eta_0, \ldots, \eta_k$ for $\HH^*(X) = \HH^*(X,\QQ)$ and let $\eta^0, \ldots, \eta^k$ denote the dual basis with respect to the Poincar\'e pairing. Without loss of generality we may suppose that $\eta^0=\mathbbm 1_X$ and $\eta^1=[Y]$. We get an induced basis $\rho_1=i^*\eta_1, \ldots, \rho_k = i^* \eta_k$ for $i^*\HH^*(X)$. Notice that $\rho_0 = i^* \eta_0 = i^* [\pt_X] = 0$, $\rho_1 = i^* \eta_1 = [\pt_Y]$.  We can extend the $\rho_i$ to a basis $\rho_1, \ldots, \rho_l$ for $\HH^*(Y)$ by adding $\rho_{k+1}\ldots,\rho_{l}$. Let $\rho^1, \ldots, \rho^l$ denote the dual basis; notice that $\rho^i$ is \emph{not} equal to $i^* \eta^i$ (they do not even have the same degree!).  Note also that $\rho^1 = \mathbbm{1}_Y$.

\subsection{Generating functions for quasimap invariants}
As with many results in enumerative geometry, the quasimap Lefschetz formula is most conveniently stated in terms of generating functions. Here we define several such generating functions for the absolute quasimap invariants of $X$ and $Y$.  We work with two marked points since this is the minimum number required in order for the quasimap space to be nonempty. However since we only take insertions at the first marking we would like to think of these, morally speaking, as $1$-pointed invariants (in Gromov--Witten theory the corresponding statement is literally true, due to the string equation).

For any smooth projective toric variety $X$ (or more generally, any space for which the quasimap invariants are defined), and any effective curve class $\beta\in \HH_2^+(X)$, we define
\begin{align*} \mathscr{S}^X_\beta(z) & =(\ev_1)_*\left(\frac{1}{z-\psi_1} \virt{\Q{0}{2}{X}{\beta}}\right) 
\intertext{and take}
S_0^X(z,q) &=\sum_{\beta\geq 0} q^\beta \mathscr{S}^X_\beta(z)\end{align*}
where $q$ is a Novikov variable and $\mathscr{S}_0^X(z) = \mathbbm 1_X$ by convention. These are generating functions for quasimap invariants of $X$ which take values in $\HH^*(X)$. We remark that each $q$-coefficient $\mathscr{S}_\beta^X(z)$ is a polynomial in $z$. 

\begin{remark} We use the notation $S_0^X(z,q)$ because this is the $\mathbf{t}=0$ restriction of the $S^{\small{0+}}$-operator applied to the fundamental class (see \S \ref{Subsection CFK comparison} below).\end{remark}

The same definition applies to $Y$. However, as noted in \S \ref{Subsection general quasimap Lefschetz}, quantum Lefschetz theorems only work to study \emph{restricted quasimap invariants}. The generating function for these is defined as
\begin{equation*} \tilde{\mathscr{S}}^Y_\beta(z) = (\ev_1)_* \left( \dfrac{1}{z-\psi_1} \virt{\Q{0}{2}{Y}{\beta}} \right) \end{equation*}
where crucially $\ev_1$ is viewed as \emph{mapping to $X$} instead of to $Y$. Thus $\tilde{\mathscr{S}}^Y_\beta(z)$ takes values in $\HH^*(X)$ and involves only quasimap invariants of $Y$ with insertions coming from $i^*\HH^*(X)$; this is in contrast to $\mathscr{S}^Y_\beta(z)$, which takes values in $\HH^*(Y)$ and involves quasimap invariants of $Y$ with arbitrary insertions. As earlier, we can also define $\tilde{S}_0^Y(z,q)$.

Now, since $X$ and $Y$ are smooth, we may use Poincar\'{e} duality to define a push-forward map on cohomology, $i_* \colon \HH^k(Y) \to \HH^{k+2}(X)$.

\begin{lem} \label{lemma pushforward} $i_* \mathscr{S}^Y_\beta(z) = \tilde{\mathscr{S}}^Y_\beta(z)$. \end{lem}
\begin{proof} This follows from functoriality of cohomological push-forwards and the fact that we have a commuting triangle:
\bcd
\Q{0}{2}{Y}{\beta} \ar[rr,"\ev_1"] \ar[rd,"\ev_1" left=0.2cm] & & Y \ar[ld,"i"] \\
& X & 
\ecd
Let us spell this out explicitly, in order to familiarise the reader with the generating functions involved. First, it is easy to see from the projection formula that:
\begin{align*} i_* \rho^i =
\begin{cases} \eta^i \qquad \text{for $i = 1, \ldots, k$} \\
0 \qquad \text{\ for $i = k+1, \ldots, l$} \end{cases} \end{align*}
Now, we can write $\mathscr{S}^Y_\beta(z)$ as:
\begin{equation*} \mathscr{S}^Y_\beta(z) = \sum_{i=1}^l \left\langle \dfrac{\rho_i}{z-\psi_1} , \mathbbm{1}_Y \right\rangle_{0,2,\beta}^Y  \rho^i \end{equation*}
Thus applying $i_*$ gives
\begin{align*} i_* \mathscr{S}^Y_\beta(z)  = \sum_{i=1}^l \left\langle \dfrac{\rho_i}{z-\psi_1} , \mathbbm{1}_Y \right\rangle_{0,2,\beta}^Y   i_* \rho^i
= \sum_{i=1}^k \left\langle \dfrac{\eta_i}{z-\psi_1}, \mathbbm{1}_X \right\rangle_{0,2,\beta}^Y   \eta^i = \tilde{\mathscr{S}}^Y_\beta(z) \end{align*}
as claimed. \end{proof}

\subsection{Quasimap quantum Lefschetz formula} \label{subsection quasimap quantum Lefschetz formula} We now turn to our main result: a formula expressing the generating function $\tilde{S}_0^Y(z,q)$ for restricted quasimap invariants of $Y$ in terms of the quasimap invariants of $X$.

\begin{thm} \label{Theorem Quantum Lefschetz}
Let $X$ and $Y$ be as above. Then
\begin{equation}\label{eqn:mirror}
\tilde{S}_0^Y(z,q) = \dfrac{\sum_{\beta\geq 0} q^\beta \left(\prod_{j=0}^{Y\cdot\beta}(Y+jz)\right) \cdot \mathscr{S}^X_\beta(z)}{P_0^X(q)}
\end{equation}
where:
\begin{align*}
 P_0^X(q) 
            & = 1 + \sum_{\substack{\beta>0 \\ K_Y\cdot\beta=0}} q^\beta(Y\cdot\beta)!\langle [\pt_X] \psi_1^{Y\cdot\beta-1} ,\mathbbm 1_{X}\rangle_{0,2,\beta}^X
\end{align*}
Notice that $P_0^X(q)$ depends not only on $X$ but also on the divisor class of $Y$ in~$X$; the superscript is supposed to indicate that the definition only involves quasimap invariants of $X$.
\end{thm}

\begin{proof}
For $m = 0, \ldots, Y \cdot \beta$, define the following generating function for $2$-pointed relative quasimap invariants
 \[
  \tilde{\mathscr{S}}_{\beta,(m)}^{X|Y}(z)=(\ev_1)_*\left(\frac{1}{z-\psi_1}\virt{\Q{0}{(m,0)}{X|Y}{\beta}}\right)
 \]
where we view $\ev_1$ as mapping to $X$.  Note that  $\tilde{\mathscr{S}}_{\beta,(0)}^{X|Y}(z) = \mathscr{S}_\beta^X(z)$. Also define the following generating function for ``comb loci invariants''
\[
 \tilde{\mathscr{R}}_{\beta,(m)}^{X|Y}(z)=(\ev_1)_*\left(m \virt{\Q{0}{(m,0)}{X|Y}{\beta}}+\frac{1}{z-\psi_1} \virt{\mathcal{D}^{\mathcal{Q}}_{(m,0),1}(X|Y,\beta)} \right)
\]
where again we view $\ev_1$ as mapping to $X$. As in \cite[Lemma 1.2]{Ga-MF}, it follows from Theorem \ref{Theorem general recursion} that
\begin{equation}\label{Gathmann recursion via generating functions}
 (Y+mz) \tilde{\mathscr{S}}_{\beta,(m)}^{X|Y}(z) = \tilde{\mathscr{S}}_{\beta,(m+1)}^{X|Y}(z)+ \tilde{\mathscr{R}}_{\beta,(m)}^{X|Y}(z)
\end{equation}
and we can apply this repeatedly to obtain:
\begin{equation} \label{eqn:G}
\prod_{j=0}^{Y\cdot\beta}(Y+jz) \mathscr{S}^X_\beta(z) = \sum_{m=0}^{Y\cdot\beta}\prod_{j=m+1}^{Y\cdot\beta}(Y+jz)\tilde{\mathscr{R}}_{\beta,(m)}^{X|Y}(z)
\end{equation}
We now examine the right-hand side in detail. By definition, $\tilde{\mathscr{R}}_{\beta,(m)}^{X|Y}(z)$ splits into two parts: those terms coming from the relative space and those terms coming from the comb loci.

Let us first consider the contribution of the comb loci. Since there are only two marked points and the first is required to lie on the internal component of the comb, it follows from the strong stability condition that there are only two options: a comb with zero teeth or a comb with one tooth.

First consider the case of a comb with zero teeth. The moduli space is then $\Q{0}{2}{Y}{\beta}$ and we require that $Y \cdot \beta = m$. Thus this piece only contributes to $\tilde{\mathscr{R}}_{\beta,(Y\cdot\beta)}^{X|Y}(z)$, and the contribution is:
\begin{equation*} \sum_{i=1}^k \left\langle \dfrac{\rho_i}{z-\psi_1}, \mathbbm{1}_Y \right\rangle_{0,2,\beta}^Y \eta^i \end{equation*}

Next consider the case of a comb with one tooth. Let $\beta^{(0)}$ and $\beta^{(1)}$ denote the curve classes of the internal and external components, respectively, and let $m^{(1)}$ be the contact order of the external component with $Y$. The picture is as follows
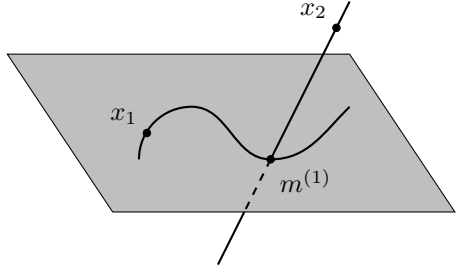
\begin{center}
\begin{tikzpicture}[scale=0.7]
  \draw [fill=gray!50] (-.5,-1) -- (6,-1) -- (4,2) -- (-2.5,2) -- (-.5,-1);
  \draw [thick] (0,0) to [out=90,in=180] (1,1) to [out=0,in=180] (2.5,0) to [out=0,in=-135] (4,1);
  \draw [thick] (2.5,0) to (4,3);
  \draw [thick,dashed] (2.5,0) to (2,-1);
  \draw [thick] (2,-1) -- (1.5,-2);
  \draw [fill] (3.75,2.5) circle [radius=2pt] node[above left]{$x_2$};
  \draw [fill] (.15,.5) circle [radius=2pt] node[above left]{$x_1$};
  \draw [fill] (2.5,0.0) circle [radius=2pt] node[below right]{$m^{(1)}$};
 \end{tikzpicture}
\end{center}
and the invariants which contribute take the form
\begin{equation*} \bigg\langle \dfrac{\rho_i}{z-\psi_1},\rho^h\bigg\rangle_{0,2,\beta^{(0)}}^Y \bigg\langle \rho_h, \mathbbm 1_{X}\bigg\rangle_{0,(m^{(1)},0),\beta^{(1)}}^{X|Y}\eta^i \end{equation*}
for $i = 1, \ldots, k$ and $h = 1, \ldots, l$. By computing dimensions, we find
\begin{align*}
0\leq \codim \rho^h &= \dim Y-\codim \rho_h \\
&= \dim Y-\vdim \Q{0}{(m^{(1)},0)}{X|Y}{\beta^{(1)}} \\
&= \dim Y-(\dim X-3-K_{X}\cdot \beta^{(1)}+2-m^{(1)})\\
&= K_Y \cdot \beta^{(1)} - Y \cdot \beta^{(1)}+m^{(1)} \\
&\leq 0
\end{align*}
where the final equality follows from adjunction and the final inequality holds because $-K_Y$ is nef and $m^{(1)}\leq Y \cdot \beta_1$. This shows that the only non-trivial contributions come from curve classes $\beta^{(1)}$ such that $K_Y \cdot \beta^{(1)}=0$, and that in this case the order of tangency must be maximal, i.e. $m^{(1)}=Y \cdot \beta^{(1)}$. Furthermore we must have $\codim \rho^h = 0$ and so $\rho^h = \rho^1 = \mathbbm{1}_Y$ which implies $\rho_h = \rho_1 = [\pt_Y]$. Finally since $m^{(1)}=Y \cdot \beta^{(1)}$ we have
\begin{equation*} m = Y \cdot \beta^{(0)}+m^{(1)}=Y \cdot (\beta^{(0)} + \beta^{(1)}) = Y \cdot \beta \end{equation*}
and so again this piece only contributes to $\tilde{\mathscr{R}}_{\beta,(Y\cdot\beta)}^{X|Y}(z)$, and the contribution is:
\begin{equation*} \sum_{i=1}^k \left( \sum_{\substack{0 < \beta^{(1)} < \beta \\ K_Y \cdot \beta^{(1)}=0}} (Y \cdot \beta^{(1)}) \bigg\langle \dfrac{\rho_i}{z-\psi_1}, \mathbbm{1}_Y \bigg\rangle_{0,2,\beta-\beta^{(1)}}^Y \bigg\langle \rho_1, \mathbbm{1}_X \bigg\rangle_{0,(Y \cdot \beta^{(1)},0),\beta^{(1)}}^{X|Y} \right) \eta^i \end{equation*}
where the $Y \cdot \beta^{(1)}$ factor comes from the weighting on the virtual class of the comb locus. Finally, we must examine the terms of $\tilde{\mathscr{R}}_{\beta,(m)}^{X|Y}(z)$ coming from:
\begin{equation*}\ev_{1*}(m\virt{\Q{0}{(m,0)}{X|Y}{\beta}})\end{equation*} 
Notice that we only have insertions from $i^*\HH^*(X) \subseteq \HH^*(Y)$, since $\ev_1$ is viewed as mapping to $X$. On the other hand
\begin{align*} \vdim \Q{0}{(m,0)}{X|Y}{\beta} & = \dim X-3 -K_X \cdot \beta +2-m & \\
& = \dim X - 1 -K_Y \cdot \beta + Y \cdot \beta - m \ \ & \text{by adjunction} \\
& \geq \dim X - 1 + Y\cdot\beta - m \ \ & \text{since $-K_Y$ is nef} \\
& \geq \dim X - 1 \ \ & \text{since $m \leq Y \cdot \beta$} \end{align*}
where in the second line we have applied the projection formula to $i$, and thus have implicitly used Assumption (2), discussed in \S \ref{Subsection setup}; namely that every curve class on $X$ comes from a class on $Y$.

Consequently the only insertions that can appear are those of dimension $0$ and $1$. However, the restriction of the $0$-dimensional class $\eta_0 = [\pt_X]$ to $Y$ vanishes, as do the restrictions of all $1$-dimensional classes except for $\eta_1$ (by the definition of the dual basis, since $\eta^1 = Y$). Thus the only insertion is $i^*\eta_1=\rho_1=[\pt_Y]$, and since $\eta^1$ has dimension $1$ all the inequalities above must actually be equalities. Thus we only have a contribution if $-K_Y \cdot \beta = 0$ and $m = Y \cdot \beta$. The contribution to $\tilde{\mathscr{R}}_{\beta,(Y\cdot\beta)}^{X|Y}(z)$ in this case is:
\begin{equation*} (Y \cdot \beta) \langle \rho_1 , \mathbbm{1}_X \rangle_{0,(Y \cdot \beta,0),\beta}^{X|Y} \eta^1 \end{equation*}

Thus we have calculated $\tilde{\mathscr{R}}_{\beta,(m)}^{X|Y}(z)$ for all $m$; substituting into equation \eqref{eqn:G} we obtain
\begin{align*} \prod_{j=0}^{Y \cdot \beta} (Y + jz) & \mathscr{S}^X_\beta(z) = \tilde{\mathscr{R}}_{\beta,(Y\cdot\beta)}^{X|Y}(z) \\
= \ & \sum_{i=1}^k \left\langle \dfrac{\rho_i}{z-\psi_1}, \mathbbm{1}_Y \right\rangle_{0,2,\beta}^Y \eta^i + \\
& \sum_{i=1}^k \left( \sum_{\substack{0 < \beta^{(1)} < \beta \\ K_Y \cdot \beta^{(1)}=0}} (Y \cdot \beta^{(1)}) \bigg\langle \dfrac{\rho_i}{z-\psi_1}, \mathbbm{1}_Y \bigg\rangle_{0,2,\beta-\beta^{(1)}}^Y \bigg\langle \rho_1, \mathbbm{1}_X \bigg\rangle_{0,(Y \cdot \beta^{(1)},0),\beta^{(1)}}^{X|Y} \right) \eta^i + \\
& (Y \cdot \beta) \langle \rho_1 , \mathbbm{1}_X \rangle_{0,(Y \cdot \beta,0),\beta}^{X|Y} \eta^1
\end{align*}
where the third term only appears if $K_Y \cdot \beta=0$. We can rewrite this as:
\begin{align*} \prod_{j=0}^{Y \cdot \beta} (Y + jz) & \mathscr{S}^X_\beta(z) \\
& = \tilde{\mathscr{S}}^Y_\beta(z) + \sum_{\substack{0 < \beta^{(1)} \leq \beta \\ K_Y \cdot \beta^{(1)}=0}} \left( (Y \cdot \beta^{(1)}) \bigg\langle \rho_1, \mathbbm{1}_X \bigg\rangle_{0,(Y \cdot \beta^{(1)},0),\beta^{(1)}}^{X|Y} \right) \tilde{\mathscr{S}}_{\beta-\beta^{(1)}}^Y(z).
\end{align*}
Summing over $\beta$, we see that equation \eqref{eqn:mirror} in the statement of Theorem \ref{Theorem Quantum Lefschetz} holds, with:
\begin{equation*} P_0^X(q) = 1 + \sum_{\substack{\beta>0 \\ K_Y\cdot\beta=0}} q^\beta (Y\cdot\beta) \langle \rho_1,\mathbbm 1_{X}\rangle_{0,(Y\cdot\beta,0),\beta}^{X|Y} \end{equation*}
To complete the proof it thus remains to show that:
\begin{equation*} P_0^X(q) = 1 + \sum_{\substack{\beta>0 \\ K_Y\cdot\beta=0}} q^\beta(Y\cdot\beta)!\langle\psi_1^{Y\cdot\beta-1} [\pt_X],\mathbbm 1_{X}\rangle_{0,2,\beta}^X \end{equation*}
The aim therefore is to express the relative invariants
\begin{equation*} \langle \rho_1 , \mathbbm{1}_X \rangle_{0,(Y\cdot\beta,0),\beta}^{X|Y} \end{equation*}
in terms of absolute invariants of $X$. Unsurprisingly, we once again do this by applying Theorem \ref{Theorem general recursion}. We have:
\begin{align*} \virt{\Q{0}{(Y \cdot \beta,0)}{X|Y}{\beta}} = \ & ((Y\cdot\beta-1)\psi_1+\ev_1^*Y) \virt{\Q{0}{(Y\cdot\beta-1,0)}{X|Y}{\beta}} \ - \\
& \virt{\mathcal{D}_{(Y\cdot\beta-1,0),1}^{\mathcal Q}(X|Y,\beta)} \end{align*}
We begin by examining the contributions from the comb loci. As before, we have only contributions coming from combs with $0$ teeth and combs with $1$ tooth. The former contributions take the form
\begin{equation*} \langle \rho_1 , \mathbbm{1}_Y \rangle_{0,2,\beta}^Y \end{equation*}
which vanish because $\vdim{\Q{0}{2}{Y}{\beta}} = \dim Y -1 -K_Y\cdot\beta = \dim Y -1$ whereas the insertion has codimension $\dim Y$. The latter contributions take the form
\begin{equation*} \langle \rho_1 ,\rho^h\rangle_{0,2,\beta^{(0)}}^Y \langle \rho_h,\mathbbm 1_{X}\rangle_{0,(Y\cdot(\beta-\beta^{(0)})-1,0),\beta-\beta^{(0)}}^{X|Y}\end{equation*}
and these must also vanish since:
\begin{align*} \codim \rho^h & = \dim Y - \codim \rho_h \\
& = \dim Y - \vdim \Q{0}{(Y\cdot(\beta-\beta^{(0)})-1,0)}{X|Y}{\beta-\beta^{(0)}} \\
& = \dim Y - (\dim X - 3 - K_X \cdot (\beta - \beta^{(0)}) + 2 - Y \cdot (\beta - \beta^{(0)}) + 1) \\
&= -1 + K_X \cdot (\beta-\beta^{(0)}) + Y \cdot (\beta-\beta^{(0)}) \\
& = -1 + K_Y\cdot(\beta-\beta^{(0)}) \\
& \leq -1
\end{align*}
Thus the comb loci do not contribute at all. Applying this recursively (the same argument as above shows that we never get comb loci contributions), we find that
\begin{align*}
(Y\cdot\beta)\langle \rho_1 ,\mathbbm 1_{X}\rangle_{0,(Y\cdot\beta,0),\beta}^{X|Y} & = (Y\cdot\beta) \langle \eta_1 \prod_{j=0}^{Y\cdot\beta-1}(Y+j\psi_1) , \mathbbm{1}_X \rangle_{0,2,\beta}^X \\
& = (Y\cdot\beta)!\langle[ \pt_X]\psi_1^{Y\cdot\beta-1},\mathbbm 1_X\rangle_{0,2,\beta}^X
\end{align*}
where the second equality holds because $Y\cdot\eta_1=\eta^1 \cdot \eta_1 = [\pt_X]$ and $Y^2\cdot\eta_1=0$. This completes the proof of Theorem \ref{Theorem Quantum Lefschetz}. \end{proof}

\begin{cor}
 If $Y$ is Fano then there is no correction term:
\begin{equation*} \sum_{\beta\geq 0} q^\beta\prod_{j=0}^{Y\cdot\beta}(Y+jz)\mathscr{S}^X_\beta(z) = \tilde{S}_0^Y(z,q) \end{equation*}
\end{cor}

\begin{cor}
Let $Y = Y_5 \subseteq  X = \PP^4$ be the quintic three-fold. Then
\begin{equation*} \tilde{S}_0^{Y_5}(z,q)=\dfrac{I_{\text{small}}^{Y_5}(z,q)}{P(q)} \end{equation*}
where
\begin{equation*} I_{\text{small}}^{Y_5}(z,q)=5H+\sum_{d>0}\frac{\prod_{j=0}^{5d}(H+jz)}{\prod_{j=0}^{d}(H+jz)^5} \ q^d \end{equation*}
and:
\begin{equation*} P(q)=1+\sum_{d>0}\frac{(5d)!}{(d!)^5}q^d. \end{equation*}
\end{cor}
\begin{proof} Apply Theorem \ref{Theorem Quantum Lefschetz} and use the fact that the quasimap invariants of $\PP^4$ coincide with the Gromov--Witten invariants, which are well-known from mirror symmetry. \end{proof}

\begin{remark}
Theorem \ref{Theorem Quantum Lefschetz} agrees with \cite[Theorem~1]{CZ-mirror} when $X$ is a projective space.
\end{remark}

\subsection{Comparison with the work of Ciocan-Fontanine and Kim} \label{Subsection CFK comparison}
Here we briefly explain how to compare our Theorem \ref{Theorem Quantum Lefschetz} to a formula obtained by Ciocan-Fontanine and Kim. We assume that the reader is familiar with the paper \cite{CF-K-wallcrossing}, in particular \S4 and \S5 thereof. There they introduce (in the more general context of $\epsilon$-stable quasimaps) the following generating functions for quasimap invariants of $Y$:
\begin{enumerate}
\item[(i)] The \emph{$J^{\epsilon}$-function}
\begin{equation*} J^\epsilon({\bf t}, z)=\sum_{m\geq 0,\beta\geq 0} \dfrac{q^\beta}{m!} (\ev_\bullet)_*\left( \prod_{i=1}^m \ev_i^*({\bf t}) \cap\operatorname{Res}_{F_0} \virt{\QGe{0}{m}{Y}{\beta}} \right) \end{equation*}
for $\mathbf{t} \in \HH^*(X)$. Here $\QGe{0}{m}{Y}{\beta}$ is the moduli space of $\epsilon$-stable quasimaps with a parametrised component \cite[\S 7.2]{CFKM}, $F_0$ is a certain fixed locus of the natural $\Gm$-action on this space, and $\ev_\bullet$ is the evaluation at the point $\infty \in \PP^1$ on the parametrised component. $\operatorname{Res}_{F_0}$ is the residue of the virtual class, i.e. the virtual class of the fixed locus divided by the Euler class of the virtual normal bundle (see \cite{GraberPandharipande} for details on virtual localisation). The variable $z$ is the $\Gm$-equivariant parameter.
\item[(ii)] The \emph{$S^\epsilon$-operator}
\begin{equation*}
 S^\epsilon(\mathbf{t},z)(\gamma)=\sum_{m\ge 0,\beta\ge 0}\frac{q^\beta}{m!} 
(\ev_1)_*\left( \dfrac{\ev_2^*(\gamma) \cdot \prod_{j=3}^{2+m} \ev_j^*({\bf t})}{z-\psi_1} \cap\virt{\Qe{0}{2+m}{Y}{\beta}} \right)
\end{equation*}
where $\mathbf{t}, \gamma \in \HH^*(X)$ and $z$ is a formal variable.
\item[(iii)] The \emph{$P^\epsilon$-series}
\begin{equation*}
 P^\epsilon({\bf t}, z)=\sum_{h=1}^k \rho^h \sum_{m\geq 0,\beta\geq 0} \frac{q^\beta }{m!} \left( \ev_1^*(\rho_h \boxtimes p_\infty) \cap \virt{\QGe{0}{1+m}{Y}{\beta}} \right) \end{equation*}
where $\mathbf{t} \in \HH^*(X)$ and $z$ is the $\Gm$-equivariant parameter. Here we view $\ev_1$ as mapping to $Y \times \PP^1$, and $p_\infty\in \HH^*_{\Gm}(\PP^1)$ is the equivariant cohomology class defined by setting $p_{\infty}|_0 =0$ and $p_{\infty}|_{\infty}=-z$.
\end{enumerate}
Given these definitions, Ciocan-Fontanine and Kim use localisation with respect to the $\Gm$-action on the parametrised space to prove the following formula \cite[Theorem 5.4.1]{CF-K-wallcrossing}:
\[
 J^\epsilon(\mathbf{t},z)=S^\epsilon(\mathbf{t},z)(P^\epsilon(\mathbf{t},z))
\]
They observe that if we set ${\bf t}=0$ and restrict to semi-positive targets, then the only class that matches non-trivially with ${P^\epsilon}|_{\mathbf{t}=0}$ is $[\pt_Y]$. Hence the above formula takes the simple form
\begin{equation} \label{Ciocan-Fontanine Theorem t zero}
 \frac{J^\epsilon |_{{\bf t}=0}}{\langle [\pt_Y],  P^\epsilon|_{{\bf t}=0}\rangle}=S^\epsilon(\mathbbm{1}_Y)|_{\mathbf{t}=0} = \mathbbm 1_Y+\sum_{h=1}^k \rho^h \left(\sum_{\beta> 0}q^\beta\left\langle\frac{\rho_h}{z-\psi},\mathbbm 1_Y\right\rangle_{0,2,\beta}^{Y,\epsilon} \right)\end{equation}
see \cite[Corollary 5.5.1]{CF-K-wallcrossing}. In our setting, $\epsilon=0+$ and $Y$ embeds as a very ample hypersurface in a toric Fano variety $X$. Our Theorem \ref{Theorem Quantum Lefschetz} makes explicit a consequence of formula \eqref{Ciocan-Fontanine Theorem t zero}. More precisely:
\begin{lem} We have the following relations between our generating functions and the generating functions of Ciocan-Fontanine and Kim:
\begin{align}
\label{J epsilon equation} i_*J^{0+}|_{\mathbf{t}=0} & = \sum_{\beta \geq 0} q^\beta \prod_{j=0}^{Y \cdot \beta} (Y + jz) \mathscr{S}^X_\beta(z) \\
\label{P epsilon equation} \langle [\pt_Y], P^{0+}|_{\mathbf{t}=0} \rangle & = P_0^X(q) \\
\label{S epsilon equation} i_*S^{0+}(\mathbbm{1}_Y)|_{\mathbf{t}=0} & = \tilde{S}^Y_0(z,q)
\end{align}
\end{lem}
\begin{proof}
\eqref{S epsilon equation} is clear from the second equality of \eqref{Ciocan-Fontanine Theorem t zero} and the definition of $\tilde{S}^Y_0(z,q)$. To show  \eqref{J epsilon equation}, let us look more closely at the left-hand side:
\begin{equation*} {J^{0+}}|_{{\mathbf{t}=0}} = \sum_{\beta\geq 0}q^\beta(\ev_\bullet)_* \left( \operatorname{Res}_{F_0}\virt{\QG{0}{0}{Y}{\beta}} \right) \end{equation*}
We have a diagram of fixed loci and evaluation maps
\bcd
\QG{0}{0}{Y}{\beta}\ar[d,hook,"i"]\ar[dr,phantom,"\Box"] & F_0^Y\ar[d, hook, "i"]\ar[l,hook']\ar[r,"\ev_{\bullet}"] & Y\ar[d,hook,"i"] \\
\QG{0}{0}{X}{\beta} & F_0^X\ar[l,hook']\ar[r,"\ev_{\bullet}"] & X
\ecd
and by a mild generalisation of \cite[Propositions 6.2.2 and 6.2.3]{CFKM}, we have an equality of $\Gm$-equivariant classes
\begin{equation*} i_*\virt{\QG{0}{0}{Y}{\beta}}=e(\pi_* E^Y_{0,0,\beta})\cap \virt{\QG{0}{0}{X}{\beta}} \end{equation*} 
where $\pi$ is the universal curve on $\QG{0}{0}{X}{\beta}$ and $E^Y_{0,0,\beta}$ is the equivariant line bundle on this curve associated to $\mathcal O_X(Y)$. This is the parametrised analogue of the bundle $L_Y$ constructed in the definition of relative quasimaps; see \S \ref{Subsection relative stable quasimaps}.
 
We would like to pull back this equation to the fixed locus $F_0^X$ in order to obtain an equation involving the residues. Let us first briefly recall the definition of $F_0^X$. Since there are no markings, any quasimap in $\QG{0}{0}{X}{\beta}$ has irreducible source curve. For such a quasimap to be $\Gm$-fixed we need that the induced rational map is constant; this means that the degree of the quasimap is concentrated at the basepoints (i.e. the sum of the lengths of the basepoints should be equal to the degree). Furthermore only the points $0$ and $\infty$ of the parametrised component are allowed to be basepoints. The fixed loci are thus indexed by ordered partitions of the degree which record the length of the basepoints at $0$ and $\infty$. $F_0^X$ is the locus on which all the degree is concentrated at $0$. This means that $\infty$ is not a basepoint and we have an evaluation map $\ev_\infty$ (denoted $\ev_{\bullet}$ earlier). See \cite[\S 4]{CF-K-wallcrossing} for more details: our $F_0^X$ is there denoted $F^{0,0,0}_{0,0,\beta}$.

Since the fibres of $\pi$ are irreducible and rational, the degree of the universal line bundle on the parametrised component is constant; therefore we have for $0 < j \leq Y\cdot\beta + 1$ an exact sequence:
\begin{equation*} 0 \to \pi_* (E^Y_{0,0,\beta}(-j\sigma_\infty)) \to \pi_* E^Y_{0,0,\beta} \to \sigma_\infty^*\mathcal{P}^{j-1}(E^Y_{0,0,\beta}) \to 0 \end{equation*}
where $\mathcal{P}^{j-1}$ denotes the bundle of $(j-1)$-jets, and $\sigma_{\infty}$ is the section given by the point $\infty \in \PP^1$ of the parametrised component. The right-hand map is given by evaluating a section of $E_{0,0,\beta}^Y$ (as well as its derivatives up to order $j-1$) at the point $\infty$. The left-hand term consists of sections of $E_{0,0,\beta}^Y$ which vanish at $\sigma_\infty$ to order $j$. If we set $j=Y\cdot\beta+1$ then this term vanishes and we have:
\begin{equation*} \pi_* E_{0,0,\beta}^Y = \sigma_\infty^* \mathcal{P}^{Y\cdot\beta}(E_{0,0,\beta}^Y) \end{equation*}
On the other hand, we have
\begin{equation*} 0 \to E^Y_{0,0,\beta} \otimes \omega_\pi^{\otimes j} \to \mathcal{P}^{j}(E^Y_{0,0,\beta}) \to \mathcal{P}^{j-1}(E^Y_{0,0,\beta}) \to 0
\end{equation*}
see \cite[\S 2]{Ga}. Pulling back along $\sigma_\infty$ and taking Euler classes, we can compute recursively from $j = Y \cdot \beta$ to $0$ and obtain a splitting
\begin{equation*}
e(\pi_* E^Y_{0,0,\beta})=\prod_{j=0}^{Y\cdot\beta} c_1(\sigma_\infty^* E^Y_{0,0,\beta}\otimes \omega_{\infty}^{\otimes j})
\end{equation*}
where $\omega_\infty=\sigma_\infty^*\omega_\pi$ gives the cotangent space at the point $\infty$. The bundle $\omega_\infty$ is (non-equivariantly) trivial since the source curves in $F_0^X$ are rigid; on the other hand the weight of the $\Gm$-action on the cotangent space at $\infty$ is $z$. We thus obtain:
\begin{equation*} i_*\virt{F_0^Y}=\prod_{j=0}^{Y\cdot\beta}(\ev_\infty^* Y+jz)\cap \virt{F_0^X} \end{equation*}
Furthermore, the Euler classes of the virtual normal bundles match under $i$. Substituting into $i_*J^{0+}|_{\mathbf{t}=0}$ we find that:
\begin{align*} i_* {J^{0+}}|_{{\mathbf{t}=0}} & = \sum_{\beta\geq 0}q^\beta(i\circ\ev_\bullet)_* \left( \operatorname{Res}_{F_0^Y}\virt{\QG{0}{0}{Y}{\beta}} \right) \\
& = \sum_{\beta \geq 0} q^\beta \prod_{j=0}^{Y \cdot \beta} (Y + jz) (\ev_{\bullet})_* \left( \operatorname{Res}_{F_0^X}\virt{\QG{0}{0}{X}{\beta}} \right) \end{align*}
On the other hand, if we apply \eqref{Ciocan-Fontanine Theorem t zero} with $X$ instead of $Y$, then the denominator on the left-hand side vanishes since $X$ is Fano. Comparing coefficients of $q^\beta$ we thus obtain
\begin{equation*} (\ev_{\bullet})_* \operatorname{Res}_{F_0^X}\virt{\QG{0}{0}{X}{\beta}} = \mathscr{S}^X_\beta(z) \end{equation*}
from which it follows that:
\begin{equation*} i_*J^{0+}|_{\mathbf{t}=0} = \sum_{\beta \geq 0} q^\beta \prod_{j=0}^{Y\cdot\beta}(Y+jz) \mathscr{S}^X_\beta(z)\end{equation*}
This proves \eqref{J epsilon equation}. It remains to show \eqref{P epsilon equation}. According to Ciocan-Fontanine and Kim, if we write the $1/z$-expansion of ${J^{\epsilon}}|_{\mathbf{t}=0}$ as
\begin{equation*} {J^{\epsilon}}|_{\mathbf{t}=0}=J^{\epsilon}_{0}(q)\mathbbm 1_Y+O(1/z) \end{equation*}
then $\langle [\pt_Y],  P^\epsilon|_{{\bf t}=0}\rangle=J^{\epsilon}_{0}(q)$. It thus remains to prove that $J^{0+}_0(q)=P_0^X(q)$.

Since $X$ is a toric Fano variety, we have the following calculation of residues due to Givental \cite{Givental-mirror} (see also \cite[Definition 7.2.8]{CF-K}):
\begin{align*}
\mathscr{S}^X_\beta(z) =\prod_{\rho\in\Sigma_X(1)}\dfrac{\prod_{j=-\infty}^0(D_{\rho}+jz)}{\prod_{j=-\infty}^{D_{\rho}\cdot \beta}(D_\rho+jz)}
=\dfrac{\prod_{\substack{\rho \colon D_\rho \cdot \beta\leq 0}} \prod_{j=D_\rho \cdot \beta}^0 (D_{\rho}+jz)}{\prod_{\substack{\rho\colon D_\rho \cdot\beta > 0}} \prod_{j=1}^{D_\rho\cdot\beta} (D_{\rho}+jz)}
\end{align*}
We can then apply equation \eqref{J epsilon equation} to find $i_*J^{0+}|_{\mathbf{t}=0}$, and hence also to find $J^{0+}_0(q)$. In the end we obtain:
\begin{equation*}
 J^{0+}_0(q)=\sum_{\beta\geq 0}q^\beta(Y\cdot\beta)!\frac{\prod_{\rho\colon D_\rho\cdot\beta< 0}(-1)^{-D_{\rho}\cdot\beta}(-D_{\rho}\cdot\beta)!}{\prod_{\rho\colon D_\rho\cdot\beta> 0}(D_{\rho}\cdot\beta)!}
\end{equation*}
On the other hand the coefficient
\begin{equation*} \langle[\pt_X]\psi_1^{Y\cdot\beta-1},\mathbbm 1_X\rangle_{0,2,\beta}^X\end{equation*}
which appears in our $P_0^X(q)$-series also appears in $\mathscr{S}^X_\beta(z)$. So again we can find it by appealing to Givental's calculation of $S_0^X(z,q)$
\begin{align*}
 \langle[\pt_X]\psi_1^{Y\cdot\beta-1},\mathbbm 1_X\rangle_{0,2,\beta}^X &=\operatorname{coeff}_{q^\beta z^{-Y\cdot\beta}}\langle [\pt_X],S_0^X(z,q) \rangle\\
& =\frac{\prod_{\rho\colon D_\rho\cdot\beta< 0}(-1)^{-D_{\rho}\cdot\beta}(-D_{\rho}\cdot\beta)!}{\prod_{\rho\colon D_\rho\cdot\beta> 0}(D_{\rho}\cdot\beta)!}
\end{align*}
which proves \eqref{P epsilon equation}. We thus conclude that \eqref{Ciocan-Fontanine Theorem t zero} implies our Theorem~\ref{Theorem Quantum Lefschetz}. \end{proof}

\section{Relative wall-crossing}\label{section wallcrossing}
\subsection{Context}
The classical mirror theorem, due to Givental, equates a certain generating function for Gromov--Witten invariants --- the $J$-\emph{function} --- with an explicit hypergeometric function --- the $I$-\emph{function} --- after a suitable change of variables called the mirror map \cite{Givental-mirror}. A fundamental insight is that the $I$-function may be interpreted as a generating function for \emph{quasimap invariants}. From this perspective, the mirror theorem breaks into two parts:
\begin{enumerate}
\item[(i)] find an explicit formula for the quasimap generating function;
\item[(ii)] prove a wall-crossing formula, relating the quasimap and Gromov--Witten generating functions via a change of variables.
\end{enumerate}
This basic strategy was pursued, with great success, in a series of papers by Ciocan-Fontanine--Kim \cite{CFKBigI,CF-K-wallcrossing,CF-K-MirrorSymmetry}.

Recently \cite{FanTsengYou} Fan--Tseng--You have used the correspondence between relative and orbifold Gromov--Witten invariants \cite{AbramovichCadmanWise} to obtain a version of the mirror theorem in the relative setting (without using quasimaps). They write down an explicit combinatorial formula for the relative $I$-function of a smooth pair $(X,Y)$, under the assumption that the pair is sufficiently semipositive and that the absolute $J$-function of $X$ is known. They then show \cite[Theorem 4.3]{FanTsengYou} that their relative $I$-function and the relative $J$-function coincide after a change of variables.

In this final section, we will show that their relative $I$-function coincides with a natural generating function $I_\beta^{X|Y}(z,0)$ for relative quasimap invariants. This provides strong evidence for our main hypothesis, namely that relative quasimap invariants provide a means to generalise the mirror theorems of Givental and Ciocan-Fontanine--Kim to the relative setting. In ongoing work in progress, we follow up on this claim by developing a fully-fledged, general theory of logarithmic quasimaps, proving reconstruction and wall-crossing formulas in this context.

\subsection{Comparison of relative $I$-functions} We begin by establishing notation. Fix as before a smooth very ample pair $(X,Y)$. We define our relative $I$-function as the following formal power series in the cohomology of $Y$
\begin{align*} I^{X|Y}(q,z,0) = \sum_\beta q^\beta \mathscr{S}_{0,(Y\cdot \beta)}^{X|Y}(z) & = \sum_\beta q^\beta (\ev_1)_*\left( \frac{1}{z-\psi_1}\virt{\Q{0}{(Y\cdot\beta,0)}{X|Y}{\beta}}\right) \\
& = \sum_\beta q^\beta \sum_{i=0}^l \left\langle \dfrac{\rho_i}{z- \psi_1}, \mathbbm{1}_X \right\rangle^{X|Y}_{0,(Y\cdot\beta,0),\beta} \rho^i \in \HH^*(Y)\llbracket z \rrbracket \end{align*}
where $\ev_1$ is viewed as mapping to $Y$ (we ignore the terms for which $Y\cdot\beta=0$). Just as in Lemma \ref{lemma pushforward} we have:
\begin{equation*} i_* I^{X|Y}(q,z,0) = \sum_\beta q^\beta \widetilde{S}_{0,(Y\cdot\beta)}^{X|Y}(z).\end{equation*}
On the other hand there is the Fan--Tseng--You relative $I$-function at $t=0$, which may be written as:
\begin{equation}\label{FTY I-fn} I_{\mathrm{FTY}}^{X|Y}(q,z,0) = \sum_\beta q^\beta i^* \left( J^X_{\beta}(z,0) \cdot \prod_{m=1}^{Y\cdot\beta-1}(Y+mz) \right) \in \HH^*(Y)\llbracket z \rrbracket. \end{equation}
See \cite[Theorem 4.3]{FanTsengYou}, and \cite[\S 7.1]{FanWuYou} for the definition of the product structure used in the first reference.
\begin{thm} \label{wallcrossing thm} Assume as in \S \ref{Subsection setup} that $Y$ is semipositive. Then we have:
\begin{equation*} i^* i_* I^{X|Y}(q,z,0) = z^{-1}\cdot i^* i_* I_{\mathrm{FTY}}^{X|Y}(q,z,0).\end{equation*}
\end{thm}
\begin{remark} The effect of applying $i^* i_*$ is to remove the terms corresponding to the non-restricted quasimap invariants. This is an artefact of the proof, since the recursion formula established in \S \ref{section recursion formula} only deals with restricted insertions. We expect that the final statement holds without this caveat.\end{remark}

\begin{proof}
It suffices to fix $\beta \in \HH_2^+(X)$ and compare the coefficients of $q^\beta$ on each side. Let us therefore do this, and set $e=Y\cdot\beta$. Applying \eqref{Gathmann recursion via generating functions} from \S \ref{subsection quasimap quantum Lefschetz formula} repeatedly, we obtain the following formula in $\HH^*(X)\llbracket z \rrbracket$:
\begin{equation*} \left( \prod_{m=0}^{e-1} (Y+mz) \right) \mathscr{S}^{X}_\beta(z) = \tilde{\mathscr{S}}_{\beta,(e)}^{X|Y}(z) + \sum_{k=0}^{e-1} \left(\prod_{m=k+1}^{e-1}(Y+mz) \right) \tilde{\mathscr{R}}^{X|Y}_{\beta,(k)}(z).\end{equation*}
By a dimension counting argument similar to the one given in the proof of Theorem \ref{Theorem Quantum Lefschetz}, we see that $\tilde{\mathscr{R}}^{X|Y}_{\beta,(k)}(z)=0$ for $k=0,\ldots,e-1$, so that the formula reads:
\begin{equation}\label{halfway to comparison of I} \tilde{\mathscr{S}}^{X|Y}_{\beta,(e)}(z) = \left( \prod_{m=0}^{e-1}(Y+mz) \right) \mathscr{S}_\beta^X(z).\end{equation}
Since $Y$ is semipositive and very ample, we have that $X$ is Fano. Hence its quasimap invariants coincide with its Gromov--Witten invariants (see \S \ref{subsection quasimaps}). The string equation then gives:
\begin{equation*} \mathscr{S}^X_{\beta}(z) = z^{-1} \cdot J^X_\beta(z,0).\end{equation*}
Applying $i^*$ to \eqref{halfway to comparison of I} we then obtain
\begin{align*} i^* \tilde{\mathscr{S}}^{X|Y}_{\beta,(e)}(z) & = z^{-1}\cdot i^* \left(J^X_\beta(z,0) \cdot \prod_{m=0}^{e-1}(Y+mz) \right) \\
& = z^{-1} \cdot i^*Y \cdot I_{\beta,\mathrm{FTY}}^{X|Y}(z,0)\\
& = z^{-1} \cdot i^* i_* I_{\beta,\mathrm{FTY}}^{X|Y}(z,0)\end{align*}
where $I_{\beta,\mathrm{FTY}}^{X|Y}(z,0)$ is the bracketed term in \eqref{FTY I-fn}. The claim follows.\end{proof}

\begin{remark} This result is unnecessarily restrictive, and is best thought of as a proof-of-concept. For one, it establishes wall-crossing only in the $t=0$ case (i.e. we only consider generating functions for two-pointed invariants with a fundamental class insertion). More importantly, the proof is not geometric; it \emph{uses} the relative mirror theorem \cite[Theorem 4.3]{FanTsengYou}, rather than providing a new proof of this result. We plan to correct both these defects, and more, in our upcoming work on logarithmic quasimap theory.
\end{remark}

\appendix

\section{Intersection-theoretic lemmas}\label{appendix:intersection}

In this appendix we explicitly define the \emph{diagonal pull-back} along a morphism whose target is unobstructed (used in \cite{Ga}) and verify that this agrees with the virtual pull-back of \cite{Manolache-Pull} when both are defined. We also check that it satisfies some expected compatibility properties.

Consider a morphism of DM stacks $f\colon Y\to X$ over a smooth base $\mathfrak M$, such that $X$ is smooth over $\mathfrak M$ and $Y$ carries a virtual class given by a perfect obstruction theory $\EE_{Y/\mathfrak M}$. Then, for every cartesian diagram
 \bcd
 G\ar[r,"g"]\ar[d,"q"]\ar[rd,phantom,"\Box"] & F\ar[d,"p"] \\
 Y\ar[r,"f"] & X 
 \ecd
and every class $\alpha\in \Achow_*(F)$, we may define
\[
 f^!_{\Delta}(\alpha)=\Delta_X^!([Y]^\text{vir}\times\alpha)\in \Achow_*(G)
\]
which we call the \emph{diagonal pull-back}. We first show that it coincides with the usual virtual pull-back along $f$ in the presence of a compatible perfect obstruction theory for $f$.

\begin{lem}\label{lem:diagonal_virtual_coincide}
Assume that there exists a relative obstruction theory $\EE_f$ compatible with $\EE_{Y/\mathfrak M}$ and the standard (unobstructed) obstruction theory for $X$, i.e:
 \bcd
 f^* \LL_{X/\mathfrak M} \ar[r] \ar[d,"\operatorname{Id}"] & \EE_{Y/\mathfrak M} \ar[r] \ar[d] & \EE_f \ar[r,"{[1]}"] \ar[d] & \, \\
 f^* \LL_{X/\mathfrak M} \ar[r] & \LL_{Y/\mathfrak M} \ar[r] & \LL_f \ar[r,"{[1]}"] & \,
 \ecd
 Then for every cartesian diagram and every class $\alpha\in \Achow_*(F)$ as above,
 \[
  f^!_{\operatorname{v}}(\alpha)=f^!_\Delta(\alpha).
 \]
\end{lem}
\begin{proof}

Consider the following cartesian diagram:
\bcd
G \ar[r,"q \times g"] \ar[d,"g"] \ar[rd,phantom,"\square"] & Y\times_{\mathfrak M}F \ar[r,"\pr_1"] \ar[d,"f \times \Id"] \ar[rd,phantom,"\square"] & Y \ar[d,"f"] \\
F \ar[r,"p \times \Id"] \ar[d,"p"] \ar[rd,phantom,"\square"] & X\times_{\mathfrak M} F \ar[r,"\pr_1"] \ar[d,"\Id \times p"] & X \\
X \ar[r,"\Delta_X"] & X\times_{\mathfrak M}X &
\ecd
Then, by commutativity of (virtual) pull-backs, we have
\begin{align*} \Delta_X^!([Y]^\text{vir}\times\alpha) & = \Delta^!((f^!_{\operatorname{v}}[X])\times\alpha) \\
& = \Delta_X^!(f^!_{\operatorname{v}}([X]\times\alpha)) \\
& = f^!_{\operatorname{v}}(\Delta_X^!([X]\times\alpha)) \\
& = f^!_{\operatorname{v}}(\alpha)
\end{align*}
as required.
\end{proof}

Secondly, we show that the diagonal pull-back behaves similarly to an ordinary virtual pull-back (e.g. commutes with other virtual pull-backs) even in the absence of a compatible perfect obstruction theory.

\begin{lem} The diagonal pull-back morphism as defined above commutes with ordinary Gysin maps and with virtual pull-backs. \end{lem}
\begin{proof} First consider the case of ordinary Gysin maps. We must consider a cartesian diagram:
\bcd
Y^{\prime \prime} \ar[r] \ar[d] \ar[rd,phantom,"\square"] & X^{\prime \prime} \ar[r] \ar[d] \ar[rd,phantom,"\square"] & S \ar[d,"k"] \\
Y^\prime \ar[r] \ar[d] \ar[rd,phantom,"\square"] & X^\prime \ar[r] \ar[d] & T \\
Y \ar[r,"f"] & X
\ecd
with $k$ a regular embedding and $f\colon Y\to X$ as before. We need to show that for all $\alpha \in A_*(X^\prime)$:
\begin{equation*} k^! f_{\Delta}^!(\alpha) = f^!_{\Delta} k^!(\alpha) \end{equation*}
We form the cartesian diagram
\bcd
Y^{\prime \prime} \ar[r] \ar[d] \ar[rd,phantom,"\square"] & Y\times X^{\prime \prime} \ar[r] \ar[d] \ar[rd,phantom,"\square"] & S \ar[d,"k"] \\
Y^\prime \ar[r] \ar[d] \ar[rd,phantom,"\square"] & Y\times X^\prime \ar[r] \ar[d] & T \\
X \ar[r,"\Delta_X"] & X\times X &
\ecd
and apply commutativity of usual Gysin morphisms. In the case where $k$ is not a regular embedding but rather is equipped with a relative perfect obstruction theory, the same argument works with $k^!$ replaced by $k_{\text{v}}^!$.
\end{proof}

\bibliographystyle{alpha}
\bibliography{relqm}

\bigskip\bigskip

\noindent Luca Battistella\\
Mathematisches Institut, Ruprecht-Karls-Universit\"at Heidelberg\\
\href{mailto:lbattistella@mathi.uni-heidelberg.de}{lbattistella@mathi.uni-heidelberg.de}\bigskip

\noindent Navid Nabijou\\
DPMMS, University of Cambridge\\
\href{mailto:nn333@cam.ac.uk}{nn333@cam.ac.uk}

\end{document}